\documentclass[11pt,a4paper,leqno,noamsfonts]{amsart}
\usepackage[a4paper,top=3cm,bottom=3cm,
    left=3cm,right=3cm,marginparwidth=80pt]{geometry}

   \makeatletter 
   \renewcommand\@biblabel[1]{#1.}
   \makeatother

\usepackage[english]{babel}
\usepackage[dvipsnames]{xcolor}
\usepackage{graphicx,pifont,soul}
\usepackage[utopia]{mathdesign}

\DeclareSymbolFont{usualmathcal}{OMS}{cmsy}{m}{n}
\DeclareSymbolFontAlphabet{\mathcal}{usualmathcal}
\DeclareMathAlphabet\BCal{OMS}{cmsy}{b}{n}

\usepackage{comment,braket}

\usepackage[shortlabels]{enumitem}
\usepackage{amsmath,amsthm,color,stmaryrd}

\usepackage{todonotes}%permette l'uso dei commenti colorati!
\usepackage{comment}%permette l'uso del comando \begin{comment}...\end{comment}
\usepackage{nicefrac}% permette la definizione del comando \nicefrac
\usepackage{indentfirst}
\usepackage[all]{xy}
\usepackage{tikz}
\usepackage{tikz-cd}
\usetikzlibrary{shapes.geometric, calc,matrix,arrows,decorations.pathmorphing,arrows.meta}
\usepackage{graphicx}
\usepackage{graphics}
\usepackage{caption}
\usepackage{extarrows}
\usepackage[enableskew, vcentermath]{youngtab}
\usepackage[utf8x]{inputenc}

\usepackage{amsmath,amsthm,color,stmaryrd}
\usepackage[all]{xy}
\usepackage[pagebackref, colorlinks=true, linkcolor=blue, citecolor=blue]{hyperref}

\usepackage{cancel}

\numberwithin{equation}{section} %numerazione delle equazioni dentro le sezioni

\theoremstyle{plain} 
\newtheorem{prop}{Proposition}[section]

\newtheorem{thm}[prop]{Theorem}
\newtheorem{lemma}[prop]{Lemma} 
\newtheorem{cor}[prop]{Corollary}

\newtheorem*{conjecture*}{Conjecture}
\newtheorem*{thm*}{Theorem}

\newtheorem{question}[prop]{Question}
\newtheorem*{question*}{Question}

\theoremstyle{remark}

\newtheorem{rmk}[prop]{Remark}
\newtheorem{example}[prop]{Example}

\newtheorem*{comment*}{Comment}

\theoremstyle{definition}
\newtheorem{defn}[prop]{Definition}

\newcommand{\CC}{\mathbb{C}}

\newcommand{\lP}{\mathbb{P}}

\newcommand{\lR}{\mathbb{R}}
\newcommand{\ZZ}{\mathbb{Z}}

\newcommand{\cA}{\mathcal{A}}

\newcommand{\cO}{\mathcal{O}}

\newcommand{\cX}{\mathcal{X}}
\newcommand{\sF}{\mathcal{F}}

\newcommand{\sA}{\mathcal{A}}

\newcommand{\qA}{\mathsf{D}}
\newcommand{\eA}{\mathsf{eD}}

\newcommand*{\defeq}{\mathrel{\vcenter{\baselineskip0.5ex \lineskiplimit0pt
\hbox{\scriptsize.}\hbox{\scriptsize.}}}%
=}

\newcommand{\ssl}{\mathfrak{sl}}
\newcommand{\su}{\mathfrak{su}}

\newcommand{\Def}{\mathsf{Def}}
\newcommand{\Ext}{\operatorname{Ext}}
\newcommand{\End}{\operatorname{End}}

\newcommand{\oH}{\operatorname{H}}

\newcommand{\Hom}{\operatorname{Hom}}
\newcommand{\RHom}{\mathsf{RHom}}
\newcommand{\MC}{\operatorname{MC}}

\newcommand{\Pic}{\operatorname{Pic}}
\newcommand{\Art}{\mathsf{Art}}
\newcommand{\Sets}{\mathsf{Sets}}

\DeclareMathOperator{\Sp}{Sp}
\DeclareMathOperator{\SU}{SU}
\DeclareMathOperator{\id}{id}

\DeclareMathOperator{\rk}{rk}

\DeclareMathOperator{\Sym}{Sym}

\newcommand{\im}{\operatorname{Im}}

\newcommand{\debar}{\bar{\partial}}

\begin{document}

\title[Hyper-holomorphic connections on HK manifolds]{Hyper-holomorphic connections on vector bundles\\
on hyper-K\"ahler manifolds}

\date{April 8, 2022}

\author{Francesco Meazzini}
\address{\newline
Alma Mater studiorum Universit\`a di Bologna, \hfill\newline
Dipartimento di Matematica\hfill\newline
Piazza di porta San Donato 5, 40126 Bologna, Italy}
\email[F.~Meazzini]{francesco.meazzini2@unibo.it}

\author{Claudio Onorati}
\address{\newline
Universit\`a degli studi di Roma Tor Vergata, \hfill\newline
Dipartimento di Matematica\hfill\newline
Viale della ricerca scientifica 1, 00133, Roma, Italy}
\email[C.~Onorati]{onorati@mat.uniroma2.it}

\begin{abstract}
We study infinitesimal deformations of autodual and hyper-holomorphic connections on complex vector bundles on hyper-K\"ahler manifolds of arbitrary dimension. In particular, we describe the DG Lie algebra controlling this deformation problem.
Moreover, we prove associative formality for derived endomorphisms of a holomorphic vector bundle admitting a projectively hyper-holomorphic connection.
\end{abstract}

%14D15 formal methods, deformations

\maketitle

\tableofcontents
%%%%%%%%%%%%%%%%%%%%%%%%%%%%%%%%%%%%%%%%%%
%%%%%%%%%%%%%%%%%%%%%%%%%%%%%%%%%%%%%%%%%%

\section*{Introduction}
The notion of formality for (either commutative, or associative, or Lie) differential graded algebras has been used in several settings in order to obtain deep geometric results.

The first big step dates back to the work by Deligne--Griffiths--Morgan--Sullivan~\cite{DGMS}, where they proved that the de Rham algebra $(A^{\ast}_{\mathbb{R}},\operatorname{d}_{\operatorname{dR}})$ of a compact K\"ahler manifold $X$ is formal, hence deducing that the real homotopy type of $X$, controlled by the homotopy class of $(A^{\ast}_{\mathbb{R}},\operatorname{d}_{\operatorname{dR}})$, is essentially determined by the
cohomology algebra $\oH^{\ast}(X, \mathbb{R})$.

Using the same tecniques of~\cite{DGMS}, Goldman and Millson~\cite{GM88} studied the moduli space of flat connections on a fixed complex vector bundle; this moduli space is hystorically a very interesting object being related to the character variety of representations of the fundamental group of the manifold and to the moduli space of Higgs bundles (which is an instance of the famous non-abelian Hodge theory). More precisely, if $\nabla$ is a flat connection on $E$, then $\nabla$ extends to a differential on the graded algebra $A^\ast(End(E))$ of complex valued $C^\infty$ forms with coefficients in the endomorphism bundle $End(E)$, whose formality implies that the moduli space of certain representations of the
fundamental group of a compact K\"ahler manifold admits at most quadratic singularities.

Another big step forward has been achieved by Kontsevich in~\cite{Kon}, where he proved that every finite dimensional Poisson manifold admits a canonical deformation quantization. This was obtained by showing the Lie formality of the Hochschild cohomology of the algebra $A$ of smooth functions on a differentiable manifold with coefficients in $A$.

More recently, the papers~\cite{BZ,BaMaMe21} dealt with the so called \emph{Kaledin-Lehn formality conjecture}, which mainly arose in the paper~\cite{KL07}. First it has been proved in~\cite{BZ} that the homotopy class of derived endomorphisms $\RHom(F,F)$ of a polystable coherent sheaf $F$ on a K3 surface is associatively formal, then in \cite{BaMaMe21} it has been shown that $\RHom(F,F)$ is Lie formal for every polystable coherent sheaf $F$ on a smooth minimal surface of Kodaira dimension $0$. In particular, both results imply that the moduli space of semistable sheaves on a K3 surface admits at most quadratic singularities. 

It is natural to ask if the same result holds for higher dimensional hyper-K\"ahler manifolds, but both the proofs mentioned above fail and the situation seems to become much more complicated. A milestone in such a higher dimensional setting is the paper~\cite{Verbitsky:JAG1996} where Verbitsky proved that the moduli space of stable vector bundles admitting a \emph{projectively hyper-holomorphic connection} on a hyper-K\"ahler manifold has at most quadratic singularities. Roughly speaking, for a vector bundle $E$ to have a projectively hyper-holomorphic connection means that the associated Azumaya algebra $End(E)$ deforms, as a holomorphic vector bundle, along the twistor family (cf.\ the definition of modular sheaf given in \cite{Markman:Modular}).

It is worth pointing out that on a K3 surface every stable holomorphic vector bundle admits such a connection, while on a higher dimensional hyper-K\"ahler manifold this turns out to be a particular restrictive assumption.

In the paper~\cite{O'Grady:Modular}, O'Grady introduced the notion of \emph{modular} sheaves, enlarging the class studied by Verbitsky to torsion free sheaves that are not necessarily locally free sheaves, and proving that for modular sheaves there exist a wall and chamber decomposition of the ample cone that behaves as in the classical K3 case.

In the literature one also finds few classes of sheaves on hyper-K\"ahler manifolds for which derived endomorphisms can be represented by a formal DG Lie algebra: the first class consists of some torsion sheaves supported on lagrangian submanifolds (see~\cite{Mladenov:Formality}), while the second consists of so-called atomic sheaves (see~\cite{Beckmann}).

One of the main focuses of the present paper is the study of the associative Dolbeault DG algebra
$\left(A^{0,\ast}(End(E)),\debar\right)$
of $C^{\infty}$-valued differential forms with coefficients in the endomorphism bundle $End(E)$ of a holomorphic vector bundle $E$ admitting a projectively hyper-holomorphic connection $\nabla$ such that $\nabla^{0,1}+\debar$. 
It is well known that $(A^{0,\ast}(End(E)),\debar)$ represents the homotopy class $\RHom(E,E)$ (see e.g.\ \cite{FIM}). We prove the following result.

\begin{thm*}[Theorem~\ref{thm:End are DGMS}]
Let $E$ be a holomorphic vector bundle on a hyper-K\"ahler manifold and assume that $E$ admits a projectively hyper-holomorphic connection. Then the Dolbeault DG algebra
\[ \left(A^{0,\ast}(End(E)),\debar\right) \]
is associatively formal.
\end{thm*}
The same result has been proved, with a different method, in \cite[Theorem~6.1]{Beckmann} for powers of slope stable vector bundles admitting a projectively hyper-holomorphic connection. Our approach is completely different and highlights the hyper-K\"ahler structure.

When the holomorphic vector bundle bundle admits a flat connection, which is an example of a hyper-holomorphic connection, the same result is proved in \cite[Section~9.2]{GM88}.
As a corollary we recover Verbitsky's quadraticity result. Our proof is easy, as it relies on the same ideas already pointed out in~\cite{DGMS,GM88}, where formality is obtained by either combining the de Rham differential $\operatorname{d}_{\operatorname{dR}}$ with its adjoint $\operatorname{d}_{\operatorname{dR}}^c$, or by showing that $\nabla$ and its $(0,1)$-part $\nabla^{0,1}$ satisfy similar properties whenever $\nabla$ is a flat connection.
Following this idea, in this paper we consider \emph{autodual connections}, of which hyper-holomorphic connections are the metric case, which can be thought of as a generalisation of flat connections since they ``become flat" after taking an appropriate quaternionic quotient (cf. Definition~\ref{def.qD}). Here is where the hyper-K\"ahler structure comes into play. 
Indeed this additional structure induces an action of the Lie group $\SU(2)$ on the real tangent space of the manifold, and hence on all of its tensor products. Using this $\SU(2)$-action one can define a special class of Yang--Mills connections that have been called \emph{autodual} by Kaledin--Verbitsky in~\cite{Kaledin.Verbitsky98}. Moreover, using Verbitsky's terminology~\cite{Verbitsky:JAG1996}, an \emph{hyper-holomorphic} connection can be defined as an autodual connection arising as the Chern connection of an hermitian metric.

Following the analogy with flat connections, we look for the DG Lie algebra controlling infinitesimal deformations of an autodual connection. As already recalled, if a connection $\nabla$ is flat over $E$, then its deformations are controlled by the DG Lie algebra $(A^\ast(End(E)),[\nabla,-])$. Since autodual connections are analogous to flat connections on the quaternionic quotient, our first guess is to use such complex. More precisely, Verbitsky in \cite{Verbitsky:Compositio2007} describes the associative \emph{quaternionic Dolbeault DG algebra} \[(\qA^\ast(End(E)),[\nabla_+,-]) \; , \]
which is the quotient of the de Rham DG algebra $A^\ast(End(E))$ with coefficients in $End(E)$ by an ideal defined using the $SU(2)$-action (cf. Definition~\ref{def.qD}). Here the assumption on $\nabla$ to be autodual ensures that the induced operator $[\nabla_+,-]$ on $\qA^{\ast}(End(E))$ squares to zero, hence providing a differential on $\qA^\ast(End(E))$.
We prove the following result.

\begin{thm*}[Theorem~\ref{thm:iso of def}]
The infinitesimal autodual deformations of an autodual connection $\nabla$ on a vector bundle $E$ are controlled by the quaternionic Dolbeault DG Lie algebra 
\[ \left(\qA^{\ast}(End(E)),[\nabla_+,-],[-,-]\right).\] 
\end{thm*}
What we mean by an \emph{autodual deformation} of an autodual connection is made precise in Section $\S~\ref{section:moduli of connections}$, where we carefully define the corresponding deformation functor (cf.\  Definition~\ref{def.deformationconnection}).

The main properties of the quaternionic Dolbeault complex are studied in $\S$~\ref{section:def of hyper-hol connections}. In Question~\ref{question:qD formal} we ask whether this DG Lie algebra is formal. When the corresponding holomorphic vector bundle has a homotopy abelian DG Lie algebra, then we remark in Proposition~\ref{prop:qD hom abelian} that this is the case for the quaternionic Dolbeault as well, hence being indeed formal. Nevertheless, the general behaviour is not known and we remark in Corollary~\ref{cor:qD not formal} that the general machinery due to \cite{DGMS} does not apply in this case, so the formality problem is more subtle.

Finally, given an autodual connection $\nabla$ on $E$, its $(0,1)$-part $\nabla^{0,1}$ defines a holomorphic structure on $E$. It is natural to wonder how the holomorphic deformations of $(E,\nabla^{0,1})$ are connected to the autodual deformations of $(E,\nabla)$. Our answer to this question can be resumed as follows.

\begin{thm*}[Corollary~\ref{cor.autodualwithfixedhol}] 
Let $E$ be a complex vector bundle on a hyper-K\"ahler manifold. Suppose $E$ is endowed with a holomorphic structure $\debar$ and a hyper-holomorphic connection $\nabla=\nabla^{1,0}+\debar$. Then
\begin{enumerate}
    \item for every holomorphic deformation $\,(E,\debar')$ of $\,(E,\debar)$ there exists an autodual connection $\nabla'$ whose $(0,1)$-part is $\debar'$,
    \item there exists a $1\colon 1$ correspondence
    \[ \left\{ \mbox{first order deformations of $\debar$} \right\}\leftrightarrow \left\{ \parbox{15em}{first order autodual deformations\\ $\mbox{ }\quad$ of $\,\nabla\,$ whose $\,(0,1)$-part is $\debar$}\right\} \]
\end{enumerate}
\end{thm*}

We conclude by making the following remark. Moduli spaces of autodual connections has been constructed analytically and studied in \cite{Kaledin.Verbitsky98}. In particular, in \cite[Section~2.5]{Kaledin.Verbitsky98} the authors construct a symplectic form on the holomorphic tangent space of such a moduli space. More generally they achieve these results for Yang--Mills connections on compact K\"ahler manifolds.
In \cite[Conjecture~8.1]{Kaledin.Verbitsky98}, the authors foresee that this symplectic form should come from an hyper-K\"ahler structure.
From our point of view, an evidence for this conjecture seems to be given by Proposition~\ref{prop.autodualVSdefor} and \cite[Section~9]{Verbitsky:JAG1996}. More precisely, in loc.\ cit.\ Verbitsky defines a (singular) hyper-K\"ahler structure on the tangent space (in the Kuranishi family) of a holomorphic vector bundle admitting a hyper-holomorphic connection and our Proposition~\ref{prop.autodualVSdefor} states that the tangent space of the corresponding hyper-holomorphic connection is isomorphic to two copies of the said tangent space. In particular it has a hyper-K\"ahler structure as well. 

\subsection*{Structure of the paper}

In Section $\S$~\ref{section:associative algebras} we provide an abstract presentation of the methods developed in~\cite{DGMS}. More precisely, we introduce a class of associative DG algebras, endowed with an extra structure miming the main ingredients used by Deligne--Griffiths--Morgan--Sullivan. For this reason we call them \emph{DGMS algebras}, see Definition~\ref{def.DGMSalgebra}. For the sake of completness we state and prove an abstract formality result for DGMS algebras (cf. Theorem~\ref{thm.DGMS}). Despite the fact that the proofs are essentially the same of~\cite{DGMS}, we found it useful to have an abstract account of the properties of this class of formal associative DG algebras, and we shall use them in different situations. We also briefly discuss the analog constructions in the Lie case.

In Section $\S$~\ref{section.claudio} we recall basic definitions and results from the theory of irreducible holomorphic symplectic manifolds and autodual connections, while Section $\S$~\ref{section:def of hyper-hol connections} is devoted to a characterisation of autodual connections in terms of Maurer--Cartan equations, cf.~Proposition~\ref{prop:hyper-hol as MC}. Here we also recall the notion of \emph{quaterionic Dolbeault DG algebra} (cf. Definition~\ref{def.qD}), and we enlight its main properties (cf. Lemma~\ref{lemma:cohom of qD}, and Theorem~\ref{thm:formality of qD}), which will play an important role later on in the paper.

In Section $\S$~\ref{section:moduli spaces holomorphic} we recall the classical theory of infinitesimal deformations of holomorphic vector bundles via DG Lie algebras and prove our formality statement, cf. Theorem~\ref{thm:End are DGMS}. We also give some examples and contextualise our result with the existing literature.

Finally, in Section $\S$~\ref{section:moduli of connections} we study local properties of the moduli space of autodual connections on a fixed complex vector bundle. More precisely, we define the deformation functor associated to an autodual connection and prove that it is naturally isomorphic to the deformation functor associated to the quaternionic Dolbeault DG Lie algebra via Maurer--Cartan equation modulo gauge action, cf. Theorem~\ref{thm:iso of def}.

%%%%%%%%%%%%%%%%%%%%%%%%%%
%%%%%%%%%%%%%%%%%%
%%%%%%%%%%%%%%%%%%%%%%%%%%
%%%%%%%%%%%%%%%%%%

\subsection*{Acknowledgements}
This work grew up from our attempt to understand Verbitsky's work on hyper-holomorphic connections, in particular his quadraticity result. We would like to thank him for his seminal papers on the subject, to which the present paper owes a lot.

We are grateful to Simone Diverio, Marco Manetti and Kieran O'Grady for useful conversations on the subject of this paper. 

Both authors acknowledge support from Ateneo 2017 \emph{Variet\`a speciali, spazi di moduli e teoria delle deformazioni}, that made possible the conference ``hyper-K\"ahler varieties in Rome'' organised in September 2021.
The second author is also supported by the PRIN grant with CUP E84I19000500006.

%%%%%%%%%%%%%%%%%%%%%%%%%%%%%%%%%%%%%%%%%%
%%%%%%%%%%%%%%%%%%%%%%%%%%%%%%%%%%%%%%%%%%
\section{DGMS algebras}\label{section:associative algebras}

In this first section we work over a field $\mathbb{K}$ of characteristic $0$. The following algebraic data are inspired by the paper \cite{DGMS}, so that the acronym DGMS in Definition~\ref{def.DGMS} stands for Deligne--Griffiths--Morgan--Sullivan.

\begin{defn}\label{def.DGMS}
A DGMS vector space $(A,d_0,d_1)$ is the datum of a graded $\mathbb{K}$-vector space $A=\bigoplus_{i\in\mathbb{Z}}A^i$ endowed with two $\mathbb{K}$-linear morphisms $d_0,d_1\in\Hom^1(A,A)$ of degree 1 satisfying the following properties:
\begin{enumerate}
\item $d_0^2=d_1^2=0$,
\item $[d_0,d_1]=d_0d_1+d_1d_0=0$,
\item the \emph{strong} $d_0d_1$\emph{-lemma} holds $\ker(d_0)\cap\ker(d_1)\cap\left(\im(d_0)+\im(d_1)\right)=\im(d_0d_1)$.
\end{enumerate}
\end{defn}

In literature (for example in \cite[Section~7]{GM88}), the strong $d_0d_1$-lemma is often referred to as "principle of the two types", since in all applications we always think at $d_0$ and $d_1$ as the $(1,0)$ and the $(0,1)$ part of a degree $1$ differential $d$ with respect to some additional structure.

\begin{lemma}\label{lemma.DGMS}
Let $A=\bigoplus\limits_{j\in\mathbb{Z}}A^j$ be a graded vector space endowed with two differentials $d_0$ and $d_1$. Under the hypothesis $[d_0,d_1]=0$, the following conditions are equivalent.
\begin{description}
\item[a] Given an element $x\in A$ which is both $d_0$-exact and $d_1$-closed, then there exists $y\in A$ such that $x=d_0d_1y$.
\item[b] $\ker(d_1)\cap \im(d_0)= \im(d_0d_1)$.
\item[c] The subcomplex $(\im(d_0), d_1)$ has trivial cohomology.
\end{description}
Similarly, the following conditions are equivalent.
\begin{description}
\item[a*] Given an element $x\in A$ which is both $d_0$-closed and $d_1$-exact, then there exists $y\in A$ such that $x=d_0d_1y$.
\item[b*] $\ker(d_0)\cap \im(d_1)= \im(d_0d_1)$.
\item[c*] The subcomplex $(\im(d_1), d_0)$ has trivial cohomology.
\end{description}
\begin{proof}
Straightforward.
\end{proof}
\end{lemma}

In the following we shall denote by $\oH_{d_1}(A)=\frac{\ker(d_1)}{\im(d_1)}$ the cohomology of $A$ with respect to $d_1$, and similarly by $\oH_{d_0}(A)=\frac{\ker(d_0)}{\im(d_0)}$ the cohomology of $A$ with respect to $d_0$.

\begin{lemma}\label{lemma.DGMScohomology}
Let $A$ be a graded vector space endowed with two differentials $d_0$ and $d_1$ such that $[d_0,d_1]=0$.
\begin{itemize}
\item If $\,\ker(d_1)\cap \im(d_0)= \im(d_0d_1)$ then the differential induced by $d_0$ on $\oH_{d_1}(L)$ is trivial.
\item If $\,\ker(d_0)\cap \im(d_1)= \im(d_0d_1)$ then the differential induced by $d_1$ on $\oH_{d_0}(L)$ is trivial.
\end{itemize}
\begin{proof}
Consider a $d_1$-closed element $x\in A$; notice that $d_0x\in \ker(d_1)\cap\im(d_0)$, being $[d_0,d_1]=0$. By hypothesis there exists $y\in A$ such that $d_0x=d_0d_1y=-d_1d_0y$, so that $d_0x\in \im(d_1)$ as required.
The second item can be checked similarly.
\end{proof}
\end{lemma}

Notice that the standard $d_0d_1$-lemma $\ker(d_0)\cap\im(d_1)=\im(d_0d_1)$ and the analogous $d_1d_0$-lemma $\ker(d_1)\cap\im(d_0)=\im(d_1d_0)$ are not equivalent in general.
The following result essentially states that the \emph{strong} $d_0d_1$\emph{-lemma} of Definition~\ref{def.DGMS} is equivalent to require both the $d_0d_1$-lemma and the $d_1d_0$-lemma.

\begin{lemma}\label{lemma.DGMSd0d1}
Let $A$ be a graded vector space endowed with two differentials $d_0$ and $d_1$. Assume that $[d_0,d_1]=0$. Then the strong $d_0d_1$-lemma
\[ \ker(d_0)\cap\ker(d_1)\cap\left(\im(d_0)+\im(d_1)\right)=\im(d_0d_1) \]
holds if and only if all the items of Lemma~\ref{lemma.DGMS} are satisfyied.
\begin{proof}
We begin by showing that if the strong $d_0d_1$-lemma holds then $\ker(d_1)\cap \im(d_0)= \im(d_0d_1)$. To this aim it is sufficient to show the only non trivial inclusion $\ker(d_1)\cap \im(d_0) \subseteq \im(d_0d_1)$. Take a $d_1$-closed element $x$, such that $x=d_0y$ for some $y\in A$. Notice that
\[ x=d_0y\in \ker(d_0)\cap\ker(d_1)\cap\left(\im(d_0)+\im(d_1)\right) \; . \]
By assumption there exists an element $z\in A$ such that $x=d_0d_1z$. Similarly one can prove that $\ker(d_0)\cap \im(d_1)= \im(d_0d_1)$.

For the converse, assume that conditions \textbf{b} and \textbf{b*} of Lemma~\ref{lemma.DGMS} are both satisfied. We need to show that
\[ \ker(d_0)\cap\ker(d_1)\cap\left(\im(d_0)+\im(d_1)\right)\subseteq\im(d_0d_1) \; . \]
Take an element $x\in A$ such that
\[ \begin{cases}
d_0x=0\\
d_1x=0\\
x=d_0y_0+d_1y_1
\end{cases} \]
Now, since $0=d_0x=d_0d_1y_1$ we have $d_1y_1\in \ker(d_0)\cap\im(d_1)$ and by condition \textbf{b} of Lemma~\ref{lemma.DGMS} there exists $z_1\in A$ such that $d_1y_1=d_0d_1z_1$. Similarly $d_0y_0\in\ker(d_1)\cap\im(d_0)$ and by condition \textbf{b*} of Lemma~\ref{lemma.DGMS} there exists $z_0\in A$ such that $d_0y_0=d_0d_1z_0$. Hence $x=d_0d_1(z_0+z_1)$ as required.
\end{proof}
\end{lemma}

For the geometric applications of this paper, we will be interested in studying DGMS vector spaces that are endowed with an associative (not necessarily commutative) product, that is compatible with the differentials.

Recall that an associative DG algebra $(A\,,\,\cdot\,,\,d)$ is a graded vector space $A=\bigoplus_{k\in\mathbb{Z}}A^k$ together with a graded product $\,\cdot\colon A^i\times A^j\to A^{i+j}$ and a differential $d\colon A^k\to A^{k+1}$ satisfying $d^2=0$ and the (graded) Leibniz rule:
\[ d(a\cdot b)=d(a)\cdot b + (-1)^{\deg(a)}a\cdot d(b) \; . \]
for every homogeneous elements $a,b\in A$.

\begin{defn}\label{def.DGMSalgebra}
An \emph{associative DGMS algebra} $(A,\cdot,d_0,d_1)$ is the datum of a DGMS vector space $(A,d_0,d_1)$ such that both $(A,\cdot,d_0)$ and $(A,\cdot,d_1)$ are associative DG algebras.
\end{defn}

\begin{example}[de Rham algebra]\label{example:de Rham}
Let $(A,\wedge, d_0=\operatorname{d}_{\operatorname{dR}}, d_1=\operatorname{d}_{\operatorname{dR}}^c)$ be the de Rham algebra of $C^{\infty}$-valued differential forms over a complex manifold endowed with the wedge product. Here $\operatorname{d}_{\operatorname{dR}}$ is the de Rham differential and $\operatorname{d}_{\operatorname{dR}}^c$ is its Hodge adjoint. Then by the $\operatorname{d}_{\operatorname{dR}}\operatorname{d}_{\operatorname{dR}}^c$-lemma (together with the Hodge theorem -- see \cite{DGMS}) and Lemma~\ref{lemma.DGMSd0d1} we have that $A$ is an associative DGMS algebra. Notice that in this case the associative product is graded commutative. 
\end{example}

A morphism of associative DG algebras is simply a morphism of DG vector spaces that preserves the products. A morphism of associative DG algebras is called a \emph{quasi}-\emph{isomorphism} if the induced morphism in cohomology is a degreewise isomorphism.

Given an associative DG algebra $(A,\cdot,d)$, for every $a,b\in A$ we have
\[ da=db=0\; \Rightarrow\; d(a\cdot b)=0, \qquad \mbox{ and } \qquad db=0\; \Rightarrow\; (da)\cdot b=d(a\cdot b) \; ,\]
so that the cohomology $\oH(A)$ inherits the product and it is an associative DG algebra (with trivial differential).

\begin{defn}\label{defn.associativeformality}
An associative DG algebra $A$ is said to be \emph{formal} if it is quasi isomorphic to its cohomology; i.e.\ if there exists another associative DG algebra $M$ together with a pair of quasi-isomorphisms of associative DG algebras
\[ (A\,,\,\cdot\,,\,d_A) \leftarrow (M\,,\,\cdot\,,\,d_M) \rightarrow (\oH_{d_A}(A)\,,\, \cdot\,,\,0) \, . \]
\end{defn}

Formality for DG algebras has many geometric consequences. For example Deligne, Griffiths, Morgan, Sullivan~\cite[Section 6]{DGMS} proved that the de Rham DG algebra $(A,\wedge,\operatorname{d}_{\operatorname{dR}})$ of a compact K\"ahler manifold $X$ described in Example~\ref{example:de Rham} is formal, hence showing that the real homotopy type of $X$ is uniquely determined by its cohomology algebra $\oH^{\ast}(X,\mathbb{R})$.

The following formality result is largely inspired to~\cite[Section 6]{DGMS}.

\begin{thm}[Formality for DGMS algebras]\label{thm.DGMS}
Let $(A,\cdot,d_0,d_1)$ be an associative DGMS algebra. Then $(A,\cdot,d_0)$ is formal.
\begin{proof}
Denote by $A_1=\ker(d_1)\subseteq A$ the subspace of $d_1$-closed elements. Notice that $A_1$ is an associative DG sub-algebra of $A$.
Now consider the following morphisms of DG algebras
\[ (A,\cdot,d_0) \xleftarrow{\iota} (A_1,\cdot,d_0) \xrightarrow{\rho} (\oH_{d_1}(A),\cdot,d_0) \; ; \]
here $\iota$ is the natural inclusion and $\rho$ is the natural projection. By Lemma~\ref{lemma.DGMScohomology} the induced differential $d_0$ on $\oH_{d_1}(A)$ is trivial, so that it only remains to be shown that both $\iota$ and $\rho$ induce isomorphisms in cohomology. \begin{description}
\item[$\iota$-step] Consider a $d_0$-closed element $x\in A$ and notice that $d_1x\in \ker(d_0)\cap\im(d_1)$. By the strong $d_0d_1$-lemma there exists $y\in A$ such that $d_1x=d_0d_1y=-d_1d_0y$. Now define $\tilde{x}=x+d_0y\in A_1$, $\iota$ maps the $d_0$-cohomology class of $\tilde{x}$ in the $d_0$-cohomology class of $x$. This proves the surjectivity of $\iota$ in cohomology. For the injectivity take $x\in A_1$ such that $\iota(x)=x$ is $d_0$-exact in $A$, so that $x\in\ker(d_1)\cap\im(d_0)$ and again by the strong $d_0d_1$-lemma we have $x=d_0d_1y$ for some $y\in A$ so that $x$ is $d_0$-exact also in $A_1$.
\item[$\rho$-step] Consider a class $[x]\in \oH_{d_1}(A)$ represented by a $d_1$-closed element $x\in A$. Notice that $d_0x\in\ker(d_1)\cap\im(d_0)$, being $[d_0,d_1]=0$. By the $d_0d_1$-lemma there exists $y\in A$ such that $d_0x=d_0d_1y$. Now consider the element $\tilde{x}=x+d_0y \in\ker(d_1)\subseteq A_1$, which has the same $d_0$-cohomology class as $x$ in $\oH_{d_1}(A)$. This proves the surjectivity of $\rho$ in cohomology. For the injectivity take $x\in A_1$ such that $\rho(x)=[x]=0$ in $(\oH_{d_1}(A),d_0)$. Then $x\in\ker(d_0)\cap\im(d_1)$ since $d_0$ is trivial on $\oH_{d_1}(A)$; therefore by the strong $d_0d_1$-lemma there exists $y\in A$ such that $x=d_0d_1y$. Hence $[x]=[x']=0$ in $\oH_{d_0}(A_1)$ as required.
\end{description}
We have proven that $(A,\cdot,d_0)$ is quasi-isomorphic to a DG algebra with trivial differential: i.e.\ $(A,\cdot,d_0)$ is formal.
\end{proof}
\end{thm}

\begin{rmk}\label{rmk:same cohomology}
The proof of Theorem~\ref{thm.DGMS} shows that the cohomologies of any associative DGMS algebra with respect to the two differentials are the same, i.e.\ $\oH_{d_0}(A)\cong \oH_{d_1}(A)$. By the symmetry of Definition~\ref{def.DGMS} we also proved that $\left(A,\cdot,d_1\right)$ is formal.
Moreover, it is straightforward to check that the isomorphism of graded vector spaces $\oH_{d_0}(A)\cong \oH_{d_1}(A)$ preserves the associative product so that $\left(A,\cdot,d_0\right)$ and $\left(A,\cdot,d_1\right)$ are quasi-isomorphic DG algebras.
\end{rmk}

The following result can be easily proved but it will lead to powerful applications.

\begin{prop}\label{prop.DGMStrick}
Let $(A\,,\,\cdot\,,\,d_0\,,\,d_1)$ be an associative DGMS algebra. Then $(A\,,\,\cdot\,,\,d_0+d_1\,,\,d_1)$ is an associative DGMS algebra.
\begin{proof}
For simplicity of exposition let us denote by $\delta_0=d_0+d_1$. It is immediate to see that
\[ \delta_0^2=d_1^2=[\delta_0,d_1]=0 \; . \]
Moreover, $\delta_0=d_0+d_1$ is still a derivation. Let us prove the strong $\delta_0d_1$-lemma. We begin by showing $\ker(\delta_0)\cap im(d_1)\subseteq im(\delta_0d_1)$. Fix an element $a\in A$ such that $\delta_0a=0$ and $a=d_1b$. Then $d_0a=\delta_0a=0$ so that by hypothesis there exists $c\in A$ such that $a=d_0d_1c=\delta_0d_1c$. By Lemma~\ref{lemma.DGMSd0d1}, we are only left with the proof of $\ker(d_1)\cap im(\delta_0)\subseteq im(\delta_0d_1)$. To this aim, fix $a\in A$ such that $d_1a=0$ and $a=\delta_0b$. Then define $a'=a-d_1b$ and notice that
\[ d_1a'=0 \qquad \mbox{ and } \qquad a'=d_0b \; . \]
Therefore there exists $c\in A$ such that $a'=d_0d_1c=\delta_0d_1c$. We have
\[ a=\delta_0d_1c+\delta_1b=d_1(b-\delta_0c)\in \ker(\delta_0)\cap im(d_1) \subseteq im(\delta_0 d_1) \]
and the statement follows.
\end{proof}
\end{prop}

%%%%%%%%%%%%%%%%%%%%%%%%%%%%%%%%%%%%%%%%%

\subsection{The Lie case}
We briefly review some constructions in the Lie case.
Fix a field $\mathbb{K}$ of characteristic $0$. Recall that a DG Lie algebra $(L,d,[-,-])$ is the data of a graded vector space $L=\bigoplus_{k\in\mathbb{Z}}L^k$ together with a graded bracket $\,[-,-]\colon L^i\times L^j\to L^{i+j}$ and a differential $d\colon L^k\to L^{k+1}$ satisfying $d^2=0$ and the following conditions
\begin{enumerate}
    \item graded Leibniz rule: $d[a,b]=[da,b] + (-1)^{\deg(a)}[a,db]$,
    \item graded skew-symmetry: $[a,b]+(-1)^{\deg(a)\deg(b)}[b,a]= 0$,
    \item graded Jacobi identity: $[a,[b,c]]=[[a,b],c]+(-1)^{\deg(a)\deg(b)}[b,[a,c]]$,
\end{enumerate}
for every homogeneous elements $a,b,c\in L$.

Notice that since char$(\mathbb{K})=0$ the above conditions also imply
\begin{itemize}
    \item $[a,a]=0$ for every homogeneous element $a$ of even degree,
    \item Bianchi identity: $[a,[a,a]]=0$ for every homogeneous element $a$ of odd degree.
\end{itemize}

\begin{example}[Dolbeault DG Lie algebra]\label{ex.DolbeaultDGLie}
Let $(E,\debar_E)$ be a holomorphic vector bundle on a complex manifold $X$. Then
\[ \left(A^{0,\ast}(End(E))\,,\,[\debar_E,-]\,,\,[-,-]\right) \]
is the Dolbeault DG Lie algebra, where $A^{0,\ast}$ is the algebra of $C^{\infty}$-valued $(0,\ast)$-forms, so that $A^{0,\ast}(End(E))$ are the $(0,\ast)$-forms with values in the endomorphism bundle $End(E)$. Here
\begin{description}
    \item[differential] $[\debar_E,-]$ is the holomorphic structure on $End(E)$ induced by the fixed holomorphic structure on $E$; %, i.e. $\debar(\alpha\otimes e)=(\debar\alpha)\otimes e$,
    \item[bracket] $[\alpha\otimes e,\beta\otimes f]=\alpha\wedge\beta\,[e,f]$, where $[e,f]$ is the usual commutator between holomorphic local sections of $End(E)$.
\end{description}
\end{example}

A morphism of DG Lie algebras is simply a morphism of DG vector spaces that preserves the brackets. We shall say that a morphism of DG Lie algebras is a \emph{quasi-isomorphism} if the induced morphism in cohomology is a degreewise isomorphism.

Given a DG Lie algebra $(L,d,[-,-])$, for every $a,b\in L$ we have
\[ da=db=0\; \Rightarrow\; d[a,b]=0, \qquad \mbox{ and } \qquad db=0\; \Rightarrow\; [da,b]=d[a,b] \; ,\]
so that the cohomology $\oH(L)$ inherits the bracket and is a DG Lie algebra (with trivial differential).

\begin{defn}
Let $L=(L,d,[-,-])$ be a DG Lie algebra. Then
\begin{enumerate}
\item A DG Lie algebra is called \emph{formal} if it is quasi isomorphic to its cohomology as a DG Lie algebra.
\item A DG Lie algebra is called \emph{homotopy abelian} if it is quasi isomorphic to an abelian DG Lie algebra.
\end{enumerate}
\end{defn}

The above definition plays a crucial role in several geometric application. As we will recall more precisely later on, the idea is that in characteristic $0$ to any DG Lie algebra it is associated a deformation problem, and quasi isomorphic DG Lie algebras induce the same deformation problem.
Hence a deformation problem corresponding to a formal (or homotopy abelian) DG Lie algebra enjoys particularly nice geometric properties, see e.g. Corollary~\ref{cor.formalityimpliesquadraticity}. 

\begin{rmk}\label{rmk.Lieformality}
Of course any associative DG algebra induces a DG Lie algebra, simply by defining the (graded) bracket as the (graded) commutator. On the other hand, as in the classical case, not every DG Lie algebra comes from an associative DG algebra. Therefore if an associative DG algebra is formal, then the induced DG Lie algebra is formal. The converse does not hold in general.
\end{rmk}

Notice that the Dolbeault DG Lie algebra of Example~\ref{ex.DolbeaultDGLie} has a bracket that comes from a natural associative product, so that we will often refer to it both as an associative DG algebra and as a DG Lie algebra.

One of the main tasks for many geometric applications to moduli spaces is to understand when the Dolbeault DG algebra is formal (either in the associative or in the Lie sense), see e.g.~\cite{AS,BaMaMe21,BZ,KL07,Zh12}. In this paper we will provide sufficient conditions for the associative Dolbeault DG algebra to be formal.

\begin{rmk}\label{rmk.formalityVShomotopyabelianity}
It is easy to see that a DG Lie algebra $L$ is homotopy abelian if and only if it is formal (in the Lie sense of Remark~\ref{rmk.Lieformality}) and the bracket induced in cohomology is trivial. In fact, if $L$ is homotopy abelian then up to quasi-isomorphisms we may assume that $L$ has trivial bracket; then $L\rightarrow \frac{L}{\im(d_L)}\leftarrow H(L)$ is a pair of quasi-isomorphisms and $A$ is formal. The converse is immediate.
\end{rmk}

One can define a DGMS \emph{Lie algebra} as a DGMS vector space $(L,d_0,d_1)$ with a Lie bracket $[-,-]$ such that both $(L,d_0,[-,-])$ and $(L,d_1,[-,-])$ are DG Lie algebras. Everything we said above can be rephrased verbatim for DGMS Lie algebras, with very little modifications.
Our choice to work with associative algebras is motivated by the applications to study the local structure of moduli spaces of vector bundles.

%%%%%%%%%%%%%%%%%%%%%%%%%%%%%%%%%%%%%%%%%
%%%%%%%%%%%%%%%%%%%%%%%%%%%%%%%%%%%%%%%%%

\subsection{Deformation functor associated to a DG Lie algebra}\label{section.DefLie}
Let us briefly recall some well known facts about the approach to infinitesimal deformation theory via DG Lie algebras.

\begin{defn}\label{def:strongMC}
Let $(L,d,[-,-])$ be a DG Lie algebra, and fix an element $x\in L^1$.
\begin{enumerate}
    \item $x$ is a \emph{classical} solution to the Maurer--Cartan equation if $dx+\frac{1}{2}[x,x]=0$.
    \item $x$ is a \emph{strong} solution to the Maurer--Cartan equation if
\[ dx=0\qquad\mbox{ and }\qquad  [x,x]=0 \; . \]
\end{enumerate}
\end{defn}

Recall that to every homotopy class of DG Lie algebras $L$ can be associated a deformation functor $\Def_L\colon \Art_{\mathbb{C}}\to \Sets$ via Maurer--Cartan equation modulus gauge action, where $\Art_{\mathbb{C}}$ denotes the category of local Artin $\mathbb{C}$-algebras with residue field $\mathbb{C}$, see e.g.~\cite{Ma09, Ma22}.

The leading principle behind this approach states that every deformation problem defined by a deformation functor $\Def\colon \Art_{\CC}\to \Sets$ is controlled by a DG Lie algebra $L$; i.e. $\Def\cong\Def_L$ (see e.g.~\cite{GM88}).
A fundamental result is the homotopy invariance (cf. e.g.~\cite[Corollary 6.6.3]{Ma22}) which says that quasi isomorphic DG Lie algebras give rise to naturally isomorphic deformation functors. This clearly implies that the notion of formality is crucial in deformation theory: indeed if a DG Lie algebra $L$ is formal then $\Def_L\cong\Def_{\oH(L)}$ so that one can reduce the study to the \emph{strong} Maurer--Cartan solutions. These are going to characterise the autodual connections (cf. Lemma~\ref{lemma:autodual are flat}).

We will not recall in detail the technical gauge action defining the deformation functor associated to a DG Lie algebra $L$. Nevertheless, the non expert reader may think of it as a perturbation (depending on the differential of $L$) of the exponential adjoint action
\[ \exp(L^0\otimes\mathfrak{m}_B)\times \MC(L^1\otimes\mathfrak{m}_B)\to \MC(L^1\otimes\mathfrak{m}_B) \qquad\qquad (e^{ad(a)}\,,\, b) \mapsto e^{ad(a)}(b)=\sum_{k\geq0}\frac{(ad(a))^k}{k!}(b) \]
for any $B\in\Art_{\CC}$. Anyway, for a DG Lie algebra with trivial differential the gauge action is precisely given by such adjoint action. Hence for a formal DG Lie algebra the deformation functor is $\Def_L\cong \Def_{\oH(L)}$ and the latter is simply given by the quadratic cone modulo the exponential adjoint action
\[ \Def_{\oH(L)}(B)=\nicefrac{\left\{\mbox{$b\in H^1(L)\otimes\mathfrak{m}_B\;\;\vert\;\; [b,b]=0$}\right\}}{\sim} \; , \qquad  \mbox{ for every } B\in\Art_{\CC} \; .\]
In the literature there exist several useful formality criteria for DG Lie algebras~\cite{Ma15,BaMaMe21}.

Another useful way to concretely understand formality is the following. If a DG Lie algebra $L$ is formal, then the maps
\[ \Def_L\left(\nicefrac{\CC[t]}{(t^k)}\right) \to \Def_L\left(\nicefrac{\CC[t]}{(t^2)}\right) \qquad \mbox{ and } \qquad \Def_L\left(\nicefrac{\CC[t]}{(t^3)}\right) \to \Def_L\left(\nicefrac{\CC[t]}{(t^2)}\right) \]
have the same image for every $k\geq3$. Essentially this means that there are only \emph{quadratic obstructions} for ``lifting'' first order deformations to higher order ones. Moreover, if the bracket induced in the first cohomology $\oH^1(L)\cong\Def_L\left(\nicefrac{\CC[t]}{(t^2)}\right)$ is trivial then this obstructions vanish and the maps above are surjective; i.e. $\Def_L$ is unobstructed.

%%%%%%%%%%%%%%%%%%%%%%%%%%%%%%%%%%%%%%%
%%%%%%%%%%%%%%%%%%%%%%%%%%%%%%%%%%%%%%%

\section{Autodual and hyper-holomorphic connections on ihs manifolds}\label{section.claudio}

For background on the general theory of irreducible holomorphic symplectic manifolds we refer to \cite{GHJ}.
\begin{defn}
A compact K\"ahler manifold $X$ is called \emph{irreducible holomorphic symplectic} if it is simply connected and $\oH^0(\Omega^2_X)=\CC\cdot\sigma$, where $\sigma$ is a holomorphic symplectic form.
\end{defn}
By definition a symplectic form is non-degenerate, therefore $\sigma$ induces a canonical isomorphism $\Theta_X\cong\Omega^1_X$, where $\Theta_X$ is the holomorphic tangent bundle of $X$. Moreover, the previous isomorphism being symplectic, we have that $\dim X=2n$. The holomorphic volume form $\sigma^n\in\oH^0(\Omega_X^{2n})$ is also non-degenerate, so that $\Omega_X^{2n}=K_X$ is trivial. In particular, by Yau's solution of Calabi's conjecture (\cite{Yau.Calabi}), any K\"ahler class is represented by the class associated to a Ricci-flat metric.
More precisely, let us write $X=(M,I)$, where $I$ is the complex structure, and let us fix a K\"ahler class $\omega$ on $X$. Then there exists a Ricci-flat metric $g$ on $M$ such that $\omega(\cdot,\cdot)=g(I(\cdot),\cdot)$. Moreover, since $(M,g,I)$ is also simply connected and irreducible symplectic, its holonomy group with respect to $g$ is the symplectic group $\Sp(n)$ (recall that $2n=\dim X$). By the holonomy principle, there exist other two complex K\"ahler structures on $(M,g)$, denoted by $J$ and $K$, that satisfies the usual quaternionic relations, i.e.\ $IJ=-JI=K$.
The corresponding K\"ahler classes are $\omega=\omega_I$, $\omega_J$ and $\omega_K$, and the symplectic form is $\sigma=\omega_J+i\omega_K$. The specification of the three complex K\"ahler structures $I$, $J$ and $K$ is called a \emph{hyper-K\"ahler structure} on $(M,g)$. If $a,b,c\in\lR$ are such that $a^2+b^2+c^2=1$, then the quasi-complex structure $L=aI+bJ+cK$ is integrable and K\"ahler. Viceversa, any complex K\"ahler structure $L$ on $(M,g)$ arising from the hyper-K\"ahler structure as before is called an \emph{induced complex structure}. The set of induced complex structures is compactly encoded in the notion of the \emph{twistor family} $\tau\colon\mathcal{X}\to\lP^1$ (cf.\ \cite[Section~3(F)]{HKLR}). The space $\mathcal{X}$ is diffeomorphic to the product $M\times S^2$ (here $S^2$ is the differential sphere underlying $\lP^1$) and has a natural integrable complex structure such that the projection $\tau$ is holomorphic. If $(a,b,c)\in S^2$, then $\tau^{-1}(a,b,c)=(M,aI+bJ+cK)$. Vice versa, the twistor family $\tau\colon\mathcal{X}\to\lP^1$ recover the hyper-K\"ahler structure on $M$ (\cite[Theorem~3.3]{HKLR}). Notice that the other projection $q\colon\mathcal{X}\to X$ is not holomorphic.

In the following we fix once and for all a K\"ahler class $\omega$ on $X=(M,I)$. By the above discussion there is then a fixed hyper-K\"ahler structure on the associated $(M,g)$, which we shorten by writing $(X,\omega)=(M,g,I,J,K)$.

The three complex structures $I$, $J$ and $K$ act on the tangent bundle $T_M$ and so also on the cotangent bundle $T_M^{\ast}$ and all of its tensor products. In particular, there is a natural action of $I$, $J$ and $K$ on the cotangent algebra $\Lambda^{\ast}T_M^{\ast}$ induced by forcing the Leibniz rule. Since $I$, $J$ and $K$ satisfy the quaternionic relations, there is a well-defined action of the Lie algebra $\mathfrak{su}(2)$ on $\Lambda^{\ast}T_M^{\ast}$. Moreover, since the Lie group $\operatorname{SU}(2)$ is simply connected, the Lie algebra action integrates to an action of the Lie group. We resume here two important remarks due to Verbitsky.  
\begin{prop}[\protect{\cite[Proposition~1.1, Proposition~1.2]{Verbitsky:JAG1996}}]\label{prop:action of SU}
Keep notations as above.
\begin{enumerate}
\item A $k$-form $\alpha$ is $\SU(2)$-invariant if and only if $k=2p$ and $\alpha$ is of type $(p,p)$ with respect to any induced complex structure.
\item The action of $\SU(2)$ on $\Lambda^{\ast}T_M^{\ast}$ commutes with the laplacian operator. In particular it  descends to an action on the cohomology $\oH^{\ast}(M,\lR)$.
\end{enumerate}
\end{prop}
We recall here the main definitions and notations used in the rest of the paper. First of all, on the K\"ahler manifold $X=(M,g,I)$ with K\"ahler class $\omega$ we consider the Lefschetz operator 
\[ L_\omega\colon\Lambda^kT^{\ast}M\longrightarrow\Lambda^{k+2}T^{\ast}M \]
given by multiplying a cohomology class by the K\"ahler class $\omega$, and its adjoint operator
\[ \Lambda_\omega\colon\Lambda^{k}T^{\ast}M\longrightarrow\Lambda^{k-2}T^{\ast}M \]
defined via the metric $g$.

If $d$ is the usual de Rham operator, then we denote by $\partial$ and $\debar$ the Dolbeault operator associated to the choosen complex structure $I$.

If $E$ is a complex vector bundle on $X=(M,g,I)$, then we denote by $\sA^k(E)$ the sheaf of complex-valued $C^\infty$ $k$-forms with coefficients in $E$, and by $A^k(E)$ the corresponding space of global sections. Furthermore, we denote by $\cA_I^{p,q}(E)$ and $A_I^{p,q}(E)$ the sheaf and space of global section with respect to the complex structure $I$.

Recall that a \emph{connection} on $E$ is a $\CC$-linear map
\[ \nabla\colon A^0(E)\longrightarrow A^1(E) \]
satisfying the Leibniz rule 
\[ \nabla(fe)=df e+f\nabla(e)\]
for every $C^\infty$ function $f$ and section $e$ of $E$. Decomposing $A^1(E)$ in its $\pm i$-eigenspaces with respect to $I$, we usually write $\nabla=\nabla^{1,0}+\nabla^{0,1}$ accordingly.

Extending $\nabla$ to a degree $1$ operator on the DG vector space $A^\bullet(E)$, one can easily check that $\nabla^2$ is in fact $A^0(E)$-linear and hence it defines an element $R=\nabla^2\in A^2(End(E))$, where $End(E)$ is the endomorphism bundle of $E$.

Recall also that a \emph{holomorphic structure} on $E$ is a $\CC$-linear map
\[ \debar_E\colon A^0(E)\longrightarrow A^{0,1}(E) \]
satisfying the analogous Leibniz rule and such that $\debar_E^2=0$. (This definition of holomorphic structure on $E$ is justified by the Koszul--Malgrange integrability theorem, \cite{Koszul-Malgrange}.)

If $(E,\debar_E)$ is a holomorphic vector bundle, a connection $\nabla$ is compatible with the holomorphic structure if $\nabla^{0,1}=\debar_E$. Finally recall that if $E$ admits a hermitian metric, then the \emph{Chern connection} on $E$ is the unique connection that is compatible with both the holomorphic structure and the hermitian metric.

The Lefschetz operators $L_\omega$ and $\Lambda_\omega$ are extended to forms with coefficients in a complex vector bundle by acting with $\operatorname{Id}_E$ on the coefficient side.

The following definition is due to Kaledin and Verbitsky (see \cite{Kaledin.Verbitsky98}).
\begin{defn}\label{defn:YM connection}
A \emph{Yang--Mills connection} $\nabla$ on $E$ is a connection such that its curvature $\nabla^2$ is of type $(1,1)$ and 
\[ \Lambda_\omega \nabla^2=0.\]
\end{defn}
\begin{rmk}
In general one can require that $\Lambda_\omega R=c \operatorname{Id}$, and in this case one talks of Yang--Mills connections of charge $c$. Our choice to restrict to Yang--Mills connections of charge $0$ will be clarified by Definition~\ref{defn:hyper-holomorphic} and \cite[Lemma~2.1]{Verbitsky:JAG1996}.
\end{rmk}
\begin{rmk}
Let $(E,\debar_E)$ be a hermitian holomorphic vector bundle, and let $\nabla$ be a connection compatible with the holomorphic structure. If $\nabla^2$ is of type $(1,1)$ and  $\Lambda_\omega \nabla^2=c\operatorname{Id}$, then $c$ must be a multiple of $\mu(E)=\deg_\omega(E)/\rk(E)$. In particular, holomorphic vector bundles having a hermitian Yang--Mills connection (compatible with the holomorphic structure) must have  first Chern class of degree $0$.
\end{rmk}

\begin{rmk}
Notice that in the literature one usually talks about Yang--Mills hermitian metrics on a holomorphic vector bundle $E$, and the associated Chern connection is referred to as the Yang--Mills connection. On the other hand, the connections in Definition~\ref{defn:YM connection} are not necessary hermitian (nor the bundle is assumed holomorphic a priori). In \cite{Kaledin.Verbitsky98} these connections are called \emph{non-hermitian Yang--Mills}, but we find the name a bit misleading, since hermitian connections are in fact of this type. When referring to a Yang--Mills connection that is the Chern connection of a Yang--Mills metric, we simply call it a \emph{hermitian Yang--Mills connection}.
\end{rmk}
\begin{rmk}
If $\nabla$ is a Yang--Mills connection on $E$, then $\nabla^{0,1}$ is a holomorphic structure on $E$ and $\nabla$ is tautologically compatible with it.
\end{rmk}
An historical-interesting set of examples of Yang--Mills connections is provided by flat connections, that is by those connections $\nabla$ such that $\nabla^2=0$. Notice that such connections are in $1-1$ correspondence with local systems on $X$. 

As we will shortly see, the next definition provides a set of examples of Yang--Mills connections that is slighter bigger than flat connections.

\begin{defn}[\protect{\cite{Verbitsky:JAG1996,Kaledin.Verbitsky98}}]\label{defn:hyper-holomorphic}
Suppose that $X=(M,I)$ is an irreducible holomorphic symplectic manifold and $(M,g,I,J,K)$ is the hyper-K\"ahler structure on $X$ associated to a K\"ahler class $\omega$. 
\begin{enumerate}
\item Let $E$ be a complex vector bundle on $X$ and $\nabla$ a connection on $E$. Then $\nabla$ is called \emph{autodual} if the curvature $\nabla^2\in A^2(End(E))$ is $\SU(2)$-invariant.
\item Let $(E,\debar_E)$ be a holomorphic vector bundle. A hermitian metric $h$ on $E$ is called \emph{hyper-holomorphic} if the Chern connection is autodual. In this case, we call the Chern connection hyper-holomorphic.
\end{enumerate}
\end{defn}
\begin{rmk}
The definition of autodual connection first appeared in \cite{Kaledin.Verbitsky98}, while the definition of hyper-holomorphic connection first appered in \cite{Verbitsky:JAG1996}. The main difference between the two is the fact that the latter is hermitian, while the former is not in general. Our formality result (see Theorem~\ref{thm:End are DGMS}) relies heavily on the metric aspects of the connection; on the other hand, our discussion on deformations in Section~\ref{section:def of hyper-hol connections} is not bounded to the metric and relates the existence of a hyper-holomorphic connection to the existence of (many) autodual connections.
\end{rmk}
\begin{rmk}
We want to stress that the notion of autoduality is associated to the choice of a hyper-K\"ahler structure, that is to the choice of a K\"ahler class $\omega$, and so one should talk of $\omega$-autoduality (and $\omega$-hyper-holomorphicity). Nevertheless, we will always work with a fixed K\"ahler structure, therefore we can safely drop the reference to $\omega$ from the notation.
\end{rmk}
\begin{example}\label{example:T}
The first non-trivial example of vector bundle having a hyper-holomorphic connection is the holomorphic tangent bundle $\Theta_X$. In fact, let $(M,g,I,J,K)$ be the hyper-K\"ahler structure associated to the fixed K\"ahler class $\omega$ on $X$. Then $\Theta_X$ comes with a natural hermitian form induced by $g$. If $\nabla_X$ is the associated Chern connection, then we need to show that $\nabla_X^2$ is $\SU(2)$-invariant. By Proposition~\ref{prop:action of SU}, this is equivalent to show that $\nabla^2_X$ is of type $(1,1)$ with respect to any induced complex structure $L$. On the other hand, the holomorphic tangent bundle $\Theta_{(M,g,L)}$ has a hermitian structure induced by $g$ as well and its curvature (which coincides with $\nabla_X^2$) is of type $(1,1)$ with respect to $L$. 
\end{example}
\begin{lemma}[\protect{\cite[Proposition~3.9]{Kaledin.Verbitsky98}}]
The following holds.
\begin{enumerate}
\item An autodual connection is a Yang--Mills connection. 
\item A hyper-holomorphic hermitian metric is a Yang--Mills hermitian metric.
\end{enumerate}
\end{lemma}
\begin{proof}
First of all, by Proposition~\ref{prop:action of SU} we have that the curvature $\nabla^2$ of an autodual connection $\nabla$ is of type $(1,1)$ with respect to any induced complex structure, in particular for $I$. Furthermore, by \cite[Lemma~2.1]{Verbitsky:JAG1996}, any $\SU(2)$-invariant $2$-forms $\alpha$ satisfies $\Lambda_\omega \alpha=0$. 

The last statement follows similarly (see \cite[Theorem~2.3]{Verbitsky:JAG1996}).
\end{proof}

By Chern--Simon theory, the Chern classes of a holomorphic vector bundle $E$ are computed in terms of the curvature of a connection $\nabla$ on $E$. In particular, if $\nabla$ is autodual, then $c_k(E)$ is $\SU(2)$-invariant for every $k$.
The following result is a partial converse.

\begin{prop}[\protect{\cite{Verbitsky:JAG1996}}]
Let $X$ be an irreducible holomorphic symplectic manifold with fixed K\"ahler form $\omega$, and let $(E,\debar_E)$ be a holomorphic vector bundle on $E$ that is indecomposable. 
If there exists a hyper-holomorphic connection on $E$, then $E$ is slope stable.

Vice versa, if $E$ is slope stable and $c_1(E)$ and $c_2(E)$ are $\SU(2)$-invariant, then there exists a connection on $E$ that is hyper-holomorphic.
\end{prop}
\begin{proof}
By the Donaldson--Uhlenbeck--Yau theorem, an indecomposable holomorphic vector bundle is slope stable if and only if it admits a unique Yang--Mills metric, and so a unique Chern connection of Yang--Mills type. Since hyper-holomorphic connections are Yang--Mills, the first statement follows. 

The last statement follows from \cite[Theorem~2.5]{Verbitsky:JAG1996}, where it is proved that, under the hypothesis of $\SU(2)$-invariance of the first two Chern classes, the Chern connection produced by the Donaldson--Uhlenbeck--Yau theorem is hyper-holomorphic.
\end{proof}

We conclude this section with the following characterisation of autodual connections. Recall that if $X$ is an irreducible holomorphic symplectic manifold and $\omega$ a K\"ahler class, then there is a twistor family $\tau\colon\cX\to\lP^1$ associated to the hyper-K\"ahler structure corresponding to $\omega$. Let $q\colon\cX\to X$ be the non-holomorphic projection.
\begin{lemma}[\protect{\cite[Lemma~5.1]{Kaledin.Verbitsky98}}]
Let $E$ be a holomorphic vector bundle and $\nabla$ an autodual connection on it. Then the connection $q^{\ast}\nabla$ on the vector bundle $q^{\ast}E$ has curvature of type $(1,1)$. In particular, $q^{\ast}\nabla^{0,1}$ is a holomorphic structure on $q^{\ast}E$.
\end{lemma}

In other words, the lemma is saying that holomorphic vector bundles on $X$ admitting an autodual connection are those for which the pullback $q^{\ast}E$ is still holomorphic. In particular $q^{\ast}E$ can be tought of as a family of holomorphic vector bundles on the twistor family, so that autodual vector bundles are those that ``deform" along the twistor line.

%%%%%%%%%%%%%%%%%%%%%%%%%%%%%%%%%%%%%%%%%
%%%%%%%%%%%%%%%%%%%%%%%%%%%%%%%%%%%%%%%%%

%\section{A Maurer-Cartan characterisation of the autodual condition}\label{section:MC for hyper-holomorphic}

%%%%%%%%%%%%%%%%%%%%%%%%%%%%%%%%%%%%%%%%%
%%%%%%%%%%%%%%%%%%%%%%%%%%%%%%%%%%%%%%%%%
\section{The quaternionic Dolbeault complex}\label{section:def of hyper-hol connections}

\subsection{A motivating example: flat connections}\label{subsection:flat connections}
We start with the following examples which we hope will serve as a motivation for what we will say next. For a more detailed account we refer to~\cite{GM88}.

Let us consider a compact K\"ahler manifold $X$ and a complex vector bundle $E$ endowed with a flat connection $\nabla$. Decomposing $\nabla$ in its $(1,0)$ and $(0,1)$ parts, according to the complex structure of $X$, we see that the flatness condition can be written as
\[ (\nabla^{0,1})^2=0\, ,\quad (\nabla^{1,0})^2=0\,\quad \mbox{ and }\quad [\nabla^{0,1},\nabla^{1,0}]=0. \]
In particular, if we think of $\nabla^{0,1}$ as a differential on the graded vector space $A^\ast(E)$, i.e. it is a holomorphic structure on $E$, then the last two conditions can be compactly stated by saying that $\nabla^{1,0}$ is a strong solution (see Definition~\ref{def:strongMC}) to the Maurer--Cartan equation in the DG Lie algebra
\[ \left(\End_{\CC}^{\ast}(A^{\ast}(E)),d=[\nabla^{0,1},-],[-,-]\right). \]

Moreover, since $\nabla$ is flat, there always exists a hermitian metric on $E$ such that $\nabla$ is the Chern connection of this metric. Using this observation, Goldman and Millson notice that the strong $\nabla^{0,1}\,\nabla^{1,0}$-lemma holds (see \cite[Proposition~7.3]{GM88}). The proof of this result is not explicitly stated in loc.\ cit.\ but can be easily filled: using the parallel hermitian metric, one can write down the usual Nakano--Kodaira identities and the Hodge decomposition, so that the claim follows as in the classical case.

From our point of view, the triple
\[ \left(A^\ast(E),\nabla^{0,1},\nabla^{1,0}\right) \]
is an example of a DGMS vector space, cf. Definition~\ref{def.DGMS}. 
One of the main results of \cite{GM88} is the statement that the infinitesimal deformations of $\nabla$ are controlled by the DG Lie algebra 
\[ \left(A^\ast(End(E)),[\nabla,-],[-,-]\right), \]
where $[-,-]$ as usual denote the graded commutator of the natural associative product (see \cite[Proposition~6.6]{GM88}). Now, by Proposition~\ref{prop.DGMStrick} and Theorem~\ref{thm.DGMS}, it follows then that the associative DG algebra $\left(A^\ast(End(E)),[\nabla,-]\right)$ is formal (in particualr it is also Lie formal) and then the obstructions to deform $\nabla$ are quadratic (see \cite[Theorem~1]{GM88} and its proof in \cite[Section~7]{GM88}).

\subsection{Autodual connections and quaternionic Dolbeault complex}\label{section:quaternionic}
Recall that the hyper-K\"ahler structure on $X$ transfers to a $\SU(2)$-action on the de Rham complex $A^{\ast}(E)$. In particular we see $A^{\ast}(E)$ as a complex $\SU(2)$-representation, which is equivalent to being a $\ssl(2)$-representation; indeed $\su(2)_{\CC}=\ssl(2)$ and that $\SU(2)$ is simply connected, so that the set of $\SU(2)$-representations is in bijection with the set of $\su(2)$-representations. We stress the fact that the action of $\SU(2)$ on $A^{\ast}(E)$ is trivial on $E$. 

The action of $\ssl(2)$ on $A^{\ast}(E)$ is easy to write down locally in terms of the actions of the generators corresponding to the three K\"ahler structures $I$, $J$ and $K$ given by the hyper-K\"ahler structure. In particular any irreducible subrepresentation of $A^k(E)$ has weight smaller or equal to $k$ and any irreducible subrepresentation of $A^{0,k}(E)$ has weight $k$. 

One can define an ideal $\mathfrak{I}=\bigoplus_k\mathfrak{I}^k$, where $\mathfrak{I}^k$ is generated by the irreducible subrepresentations of $A^k(E)$ of weight strictly smaller that $k$. We refer to \cite{Verbitsky:Compositio2007} for the details on the definition of $\mathfrak{I}$ and the fact that it is well-defined and well-behaved.

We shall denote the quotient by
\[
A_+^{\ast}(E)\defeq\frac{A^{\ast}(E)}{\mathfrak{I}}.
\]

Notice that the Dolbeault complex $A^{0,\ast}(E)$ canonically embeds as a graded vector space in $A_+^{\ast}(E)$, since the quotient inherits the bi-graded structure $A_+^{k}(E)=\bigoplus_{p+q=k}A_+^{p,q}(E)$ and we have already remarked that any subrepresentation of $A^{0,k}(E)$ has weight exactly $k$.

Any connection $\nabla$ on $E$ induces a degree $1$ operator $\nabla_+$ on the graded algebra $A_+^\ast(E)$.

\begin{lemma}\label{lemma:autodual are flat}
$\nabla$ is autodual if and only if $\,\nabla_+$ is flat, i.e. $\,\nabla_+^2=0$.
\begin{proof}
First of all, by Proposition~\ref{prop:action of SU}, if $\nabla$ is autodual, then the curvature $\nabla^2$ is $\SU(2)$-invariant. In particular, it defines a $1$-dimensional subrepresentation of weight $0$ and hence it must belong to the ideal $\mathfrak{I}$. It follows that $\nabla_+^2=0$.

For the converse, notice that $\nabla^2$ is a $2$-form, hence the representation generated by it must have weight either $0$ or $2$. Since $\nabla_+^2=0$, it cannot have weight $2$ and therefore it must be $\SU(2)$-invariant.
\end{proof}
\end{lemma}

Let us define $\qA^{\ast}(E)=\bigoplus_k\qA^{k}(E)$, where 
\[ 
\qA^{k}(E)=\Sym^kV\otimes_{\CC}A^{0,k}(E),
\]
where $V$ is the irreducible $2$-dimensional $\ssl(2)$-representation.
As a vector space, $\qA^k(E)$ is nothing but the direct sum of $k+1$ copies of $A^{0,k}(E)$. Moreover, the choice of a basis $\{x,y\}$ for $V$ provides a bi-grading of $\qA^{\ast}(E)$
\[ \qA^k(E)=\bigoplus_{p+q=k}\qA^{p,q}(E) =\bigoplus_{p+q=k} x^py^q A^{0,p+q}(E) \; , \]
where we are thinking of $x$ and $y$ as formal central variables.

Recall that if $L$ is any induced complex structure on $(M,g,I,J,K)$, then $L$ acts on $A^{\ast}$ in a multiplicative way, i.e.\ $L(\alpha\wedge\beta)=L(\alpha)\wedge L(\beta)$. Our next result is due to Verbitsky and relates the two constructions introduced above.

\begin{lemma}\label{lemma:Phi}
There is a $\SU(2)$-equivariant isomorphism of bi-graded vector spaces
\[
\varphi\colon \qA^{\ast,\ast}(E)\longrightarrow A_+^{\ast,\ast}(E) \; .
\]
\begin{proof}
First notice that $A_+^{k}(E)=\bigoplus\limits_{p+q=k}A_+^{p,q}(E)$, and by~\cite[Proposition 2.7]{Verbitsky:Compositio2007} we have $ A_+^{p,q}\cong A^{0,p+q}$, so that in particular $A_+^{k}(E)$ is isomorphic to $k+1$ copies of $A^{0,p+q}(E)$ as a vector space. Now, fix a basis $\{x,y\}$ of the irreducible $2$-dimensional $\ssl(2)$-representation $V$, and consider the isomorphism
\[ \varphi^1\colon \qA^1(E)= V\otimes A^{0,1}(E) \to A_+^1(E) \]
defined by
\[ x\otimes(\eta\otimes e)\mapsto J(\eta)\otimes e \qquad \mbox{ and } \qquad y\otimes\eta\mapsto \eta\otimes e \; . \]
Notice that $\varphi^1$ extends, thanks to the algebra structure and the multiplicativity action of $\SU(2)$, to a map
\[ \varphi^k\colon \qA^k(E)=\Sym^k(V)\otimes A^{0,k}(E) \longrightarrow A^k_+(E) \; , \]
hence providing a morphism of graded algebras $\varphi\colon \qA^{\ast}(E) \to A_+^{\ast}(E)$,
which is the required $\SU(2)$-equivariant isomorphism, cf.\ \cite[Proposition 2.9]{Verbitsky:Compositio2007}. The inverse $\varphi^{-1}$ is defined by
\[ [\alpha\wedge\beta]\otimes e \mapsto x^{\deg(\alpha)}y^{\deg(\beta)}\left(J(\alpha)\wedge\beta\right)\otimes e\; , \]
for every homogeneous class $[\alpha\wedge\beta]\otimes e\in A_+^{\deg(\alpha),\deg(\beta)}(E)$. Here we are again tacitely using \cite[Proposition 2.7]{Verbitsky:Compositio2007} to think the class of a $(p,q)$-form in $A^{\ast}_+(E)$ as a $(0,p+q)$-form. %and the fact that $J(\alpha)$ is a $(0,\deg(\alpha))$-form obtained extending the action of $J$ multiplicatively to $J\colon A^{p,0}\to A^{0,p}$. 
\end{proof}
\end{lemma}

\begin{rmk}\label{rmk.qDdoublecomplex}
When $E$ is endowed with an autodual connection, the explicit isomorphism $\varphi$ of Lemma~\ref{lemma:Phi} together with Lemma~\ref{lemma:autodual are flat} allow us to induce two differentials on $\qA^{\ast,\ast}(E)$. In fact for every $p,q\geq0$:
\begin{enumerate}
    \item $\nabla_+^{1,0}$ is naturally taken to the differential operator
    \[ x\nabla^{0,1}_J\colon x^py^q\,A^{0,p+q}_I(E) \xrightarrow{J\otimes\id}x^py^qA^{p+q,0}_I(E)\xrightarrow{x\nabla^{1,0}} x^{p+1}y^qA_I^{p+q+1,0}(E)\xrightarrow{J^{-1}\otimes\id} x^{p+1}y^qA_I^{0,p+q+1}(E) \]
    so that
    \[ x\nabla_J^{0,1} \defeq  (J^{-1}\otimes\id)\circ(x\,\nabla^{1,0})\circ(J\otimes\id) = \varphi^{-1}\circ\nabla_+^{1,0}\circ \varphi \; , \]
    \item $\nabla_+^{1,0}$ is naturally taken to the differential operator
    \[ y\nabla^{0,1}\colon x^py^qpA^{0,p+q}(E)\longrightarrow x^py^{q+1}A^{0,p+q+1}(E) \]
    so that
    \[ y\nabla^{0,1}= \varphi^{-1}\circ\nabla_+^{0,1}\circ \varphi \; . \]
\end{enumerate}
Notice that by Lemma~\ref{lemma:autodual are flat}, assuming the connection to be autodual is equivalent to require that $\qA^{\ast,\ast}(E)$ is a double complex, i.e.
\[ (\nabla^{0,1})^2\equiv 0 \qquad (\nabla_J^{0,1})^2\equiv 0 \qquad \nabla^{0,1}\nabla_J^{0,1}+\nabla_J^{0,1}\nabla^{0,1}\equiv 0 \; . \]
\end{rmk}

\begin{defn}[quaternionic Dolbeault complex]\label{def.qD}
Let $(M,g,I,J,K)$ be a hyper-K\"ahler manifold. Let $E$ be a complex vector bundle on $X=(M,I)$ equipped with an autodual connection $\nabla=\nabla^{1,0}+\nabla^{0,1}$. 
We shall call $(\qA^{\ast}(E)\,,\,x\nabla^{0,1}_J+y\nabla^{0,1})$ the \emph{quaternionic Dolbeault complex}.
\end{defn}

\begin{rmk}
Notice that the quaternionic Dolbeault complex of Definition~\ref{def.qD} can be also seen as the total complex of the double complex
\[ \qA^{\ast,\ast}(E) =  \left(\bigoplus_{k\in\mathbb{N}}\bigoplus_{p+q=k}\qA^{p,q}(E)\,,\,x\nabla^{0,1}_J\, ,\, y\nabla^{0,1}\,\right) \; . \]
\end{rmk}

For reference purposes we explicitly state the next result, which essentially says that the isomorphism of Lemma~\ref{lemma:Phi} respects the double complex structure.

\begin{prop}\label{prop:Phi}
Suppose that $\nabla$ is an autodual connection on $E$.
There is a $\SU(2)$-equivariant isomorphism of double complexes
\[
\varphi\colon \left(\qA^{\ast,\ast}(E)\,,\,x\nabla^{0,1}_J\,,\,y\nabla^{0,1}\right)\longrightarrow \left(A_+^{\ast,\ast}(E)\,,\,\nabla_+^{0,1}\,,\,\nabla_+^{0,1}\right) \; .
\]
\begin{proof}
Immediate from Lemma~\ref{lemma:autodual are flat}, Lemma~\ref{lemma:Phi} and Remark~\ref{rmk.qDdoublecomplex}.
\end{proof}
\end{prop}

\begin{prop}\label{prop:hyper-hol as MC}
If a connection $\nabla$ on $E$ is autodual then
\begin{enumerate}
    \item $\nabla^{0,1}$ is a holomorphic structure, and
    \item $\nabla^{0,1}_J= (J^{-1}\otimes\id)\circ\nabla^{1,0}\circ (J\otimes\id)$ is a strong solution of the Maurer--Cartan equation in the DG Lie algebra $\left(\End_{\CC}^{\ast}(A^{0,\ast}(E)),d=[\nabla^{0,1},-],[-,-]\right)$.
\end{enumerate}
\begin{proof}
The statements are equivalent to the relations
\[ \left(\nabla^{0,1}\right)^2\equiv 0 \qquad \left(\nabla_J^{0,1}\right)^2\equiv 0 \qquad \nabla^{0,1}\nabla_J^{0,1}+\nabla_J^{0,1}\nabla^{0,1}\equiv 0 \]
already outlined in Remark~\ref{rmk.qDdoublecomplex}.
\end{proof}
\end{prop}

\begin{rmk}\label{rmk:extension connection}
Notice that the proposition above does not say that a solution of the Maurer--Cartan equation in $\left(\End_{\CC}^{\ast}(A^{0,\ast}(E)),[\nabla^{0,1},-],[-,-]\right)$ gives rise to an autodual connection. In fact, given such a solution $\delta$, this is true if and only if the operator $\nabla_\delta\defeq J\circ\delta\circ J^{-1}+\nabla^{0,1}$ is a connection on $E$. 
\end{rmk}

When moreover the connection $\nabla$ is hyper-holomorphic, i.e.\ it is the Chern connection of a hermitian metric, we have the further result due to Verbitsky. Similarly to before, let us define the differential $\nabla^{0,1}_J$ on $A^{0,\ast}(E)$ as the composition
\[ \nabla^{0,1}_J=(J\otimes\id)^{-1}\circ\nabla^{1,0}\circ(J\otimes\id). \]

\begin{lemma}[\protect{\cite[Theorem~4.4]{Verbitsky:JAG1996}}]\label{lemma:debardebarJlemma}
Assume that the connection $\nabla$ is hyper-holomorphic. Then the strong $\nabla^{0,1}\,\nabla^{0,1}_J$-lemma holds on $A^{0,\ast}(E)$.

In particular $\left(A^{0,\ast}(E), \nabla^{0,1}, \nabla^{0,1}_J\right)$ is a DGMS vector space.
\begin{proof}
In \cite[Theorem~4.4]{Verbitsky:JAG1996}, Verbitsky proves a real version of the $\nabla^{0,1}\,\nabla^{0,1}_J$-lemma, but the proof holds verbatim in this case.

More precisely, Verbitsky proves analogs of the classical Hodge--Nakano--Kodaira identities, so that the proof is symmetric in $\nabla^{0,1}$ and $\nabla^{0,1}_J$, i.e.\ the $\nabla^{0,1}_J\,\nabla^{0,1}$-lemma holds as well. Therefore, the claim follows by Lemma~\ref{lemma.DGMSd0d1}.
%The idea is the same as in the flat case considered in Section~\ref{subsection:flat connections}: using the hermitian metric to get a version of the Kodaira--Nakano identities and the Hodge relations. Nevertheless, we want these identities to hold between $\nabla^{0,1}$ and $\nabla^{0,1}_J$, so that it is normal to look at them at the level of the quaternionic Dolbeault complex, where the connection "becomes flat". This task has already been taken by Verbitsky, who proves the Kodaira--Nakano identities in \cite[Theorem~3.5]{Verbitsky:Compositio2007} and the Hodge relations in \cite[Theorem~3.6]{Verbitsky:Compositio2007}.
%With these results at hand, one can prove the $\nabla^{0,1}\,\nabla^{0,1}_J$-lemma as in \cite[Theorem~4.4]{Verbitsky:JAG1996} (where Verbitsky proves a slightly different version of the statement). 

%Finally, to have the strong $\nabla^{0,1}\,\nabla^{0,1}_J$-lemma, by Lemma~\ref{lemma.DGMSd0d1} it is enough to prove the $\nabla^{0,1}_J\,\nabla^{0,1}$-lemma too. Using the Hodge decomposition and the Green operator, this can be deduced from the $\nabla^{0,1}\,\nabla^{0,1}_J$-lemma (see again the proof of \cite[Theorem~4.4]{Verbitsky:JAG1996}) and we have done.
\end{proof}
\end{lemma}

%As a corollary we get the following statement, that can be thought as a compact way to summarise the results in \cite[Section~4]{Verbitsky:JAG1996}. 

%\begin{cor}\label{cor.qDisDGMSvectorspace}
%Assume that $\nabla$ is a hyper-holomorphic connection on a complex vector bundle $E$. Then the Dolbeault vector space $\left(A^{0,\ast}(E), \nabla^{0,1}, \nabla^{0,1}_J\right)$ is a DGMS vector space.
%\begin{proof}
%According to Definition~\ref{def.DGMS}, we need to check that 
%\[ (\nabla^{0,1})^2=0\, ,\qquad (\nabla^{0,1}_J)^2=0\,\qquad [\nabla^{0,1},\nabla^{0,1}_J]=0 \]
%and the strong $\nabla^{0,1}\,\nabla^{0,1}_J$-lemma holds.
%All these conditions follow at once from Remark~\ref{rmk.qDdoublecomplex} and Lemma~\ref{lemma:debardebarJlemma}.
%\end{proof}
%\end{cor}

%%%%%%%%%%%%%%%%%%%%%%%%%%%%%%%%%

\subsection{Properties of the quaternionic Dolbeault complex}
Let $E$ be a complex vector bundle on an irreducible holomorphic symplectic manifold $X$ and suppose that $\nabla=\nabla^{1,0}+\nabla^{0,1}$ is a hyper-holomorphic connection on $E$. In the rest of this section we will need the DGMS structure on the quaternionic Dolbeault complex, for this reason we restrict to hyper-holomorphic connections.

In this subsection we work with the quaternionic Dolbeault complex, see Definition~\ref{def.qD}.
Denote by $V$ the irreducible $2$-dimensional $\ssl(2)$-representation, and fix a basis $\{x,y\}$ of $V$. Then the quaternionic Dolbeault complex can be written as
\[ \left(\qA^{\ast}(E)\,,\,x\nabla^{0,1}_J+y\nabla^{0,1}\right) \]
where $\qA^{k}(E)=A^{0,k}(E)\otimes_{\CC}\CC[x,y]_k$. Here $\CC[x,y]_k\cong\Sym^k(V)$ denotes the set of homogeneous polynomials of degree $k$ in the formal central variables $x$ and $y$.

\begin{lemma}\label{lemma:cohom of qD}
Let $X$ be an irreducible holomorphic symplectic manifold. Consider a complex vector bundle $E$ on $X$, endowed with an hyper-holomorphic connection $\nabla=\nabla^{1,0}+\nabla^{0,1}$. Then
\[
\oH^k(\qA^{\ast}(E))\,\cong\,\oH^k_{y\nabla^{0,1}}(A^{0,\ast}(E))\otimes_{\CC}\CC[x,y]_k\,\cong\,\oH^k_{x\nabla^{0,1}_J}(A^{0,\ast}(E))\otimes_{\CC}\CC[x,y]_k.
\]
\end{lemma}
\begin{proof}
The proof is a standard spectral sequence argument. In fact, by Remark~\ref{rmk.qDdoublecomplex} the quaternionic Dolbeault complex is the total complex associated to the double complex
\[ \left( \qA^{\ast,\ast}(E)\,,\,x\nabla^{0,1}_J\,,\, y\nabla^{0,1} \right) \; . \]
Taking cohomology with respect to the vertical differential $y\nabla^{0,1}$, the first page $E_1$ reads like
\[ E_1^{p,q} \,=\, \begin{cases}
x^py^q\, \oH^{p+q}_{\nabla^{0,1}}(A^{0,\ast}(E)) & \mbox{ if } q\geq 1\\
x^p\,\ker \left(\nabla^{0,1}\colon A^{0,p}\to A^{0,p+1}\right) & \mbox{ if } q=0
\end{cases} \]
while the induced horizontal differentials vanish by Lemma~\ref{lemma.DGMScohomology} as soon as $q\geq1$. Therefore the second page $E_2$ will be
\[ E_2^{p,q} \,=\, \begin{cases}
x^py^q\, \oH^{p+q}_{\nabla^{0,1}}(A^{0,\ast}(E)) & \mbox{ if } q\geq 1\\
x^p\, \oH^p_{\nabla_J^{0,1}}\left(\ker \left(\nabla^{0,1}\colon A^{0,\ast}\to A^{0,\ast+1}\right)\right) & \mbox{ if } q=0 \; .
\end{cases} \]
Now we want to show that
\[ E_2^{p,0}=\oH^p_{\nabla_J^{0,1}}\left(\ker \left(\nabla^{0,1}\colon A^{0,\ast}\to A^{0,\ast+1}\right)\right) \cong \oH^p_{\nabla^{0,1}_J}\left(A^{0,\ast}(E)\right) \cong \oH^p_{\nabla^{0,1}}\left(A^{0,\ast}(E)\right) \; . \]
The last isomorphism immediately follows by Remark~\ref{rmk:same cohomology}. Let us show the former. Suppose that a class $[a]\in \oH^p_{\nabla_J^{0,1}}\left(\ker \nabla^{0,1}\right)$ is trivial, then $a=\nabla_J^{0,1}b$ for some $b\in A^{0,p-1}$ such that $\nabla^{0,1}b=0$. In particular, the map
\[ \mathsf{g}\colon \oH^p_{\nabla_J^{0,1}}\left(\ker \left(\nabla^{0,1}\colon A^{0,\ast}\to A^{0,\ast+1}\right)\right) \longrightarrow \oH^p_{\nabla^{0,1}_J}\left(A^{0,\ast}(E)\right) \qquad \qquad [a]\mapsto [a] \]
is well-defined. To show the surjectivity of $\mathsf{g}$, consider a class $[a]\in\oH^p_{\nabla^{0,1}_J}\left(A^{0,\ast}(E)\right)$ and notice that $\nabla^{0,1}a\in \ker(\nabla_J^{0,1})\cap im(\nabla^{0,1})$. Hence there exists an element $b\in A^{0,p-1}$ such that $\nabla^{0,1}a=\nabla^{0,1}\nabla^{0,1}_Jb$, so that it is sufficient to consider $a'=a-\nabla_J^{0,1}b$ which satisfies
\[ [a']\in\oH^p_{\nabla_J^{0,1}}\left(\ker \left(\nabla^{0,1}\colon A^{0,\ast}\to A^{0,\ast+1}\right)\right) \mbox{ and } [a']=[a]\in\oH^p_{\nabla^{0,1}_J}\left(A^{0,\ast}(E)\right) \; . \]
For the injectivity, let us consider an element $a\in A^{0,p}(E)$ such that $\nabla^{0,1}a=0$ and $\mathsf{g}[a]=0$. This is equivalent to say that
\[ \begin{cases}
\nabla^{0,1}a=0\\
\nabla^{0,1}_Ja=0\\
\exists b\in A^{0,p-1}(E) \mbox{ such that } a=\nabla^{0,1}_Jb
\end{cases} \]
so that $\nabla^{0,1}b\in \ker(\nabla^{0,1}_J)\cap im(\nabla^{0,1})$. Therefore there exists $c\in A^{0,p-2}(E)$ such that $\nabla^{0,1}b=\nabla^{0,1}\nabla^{0,1}_Jc$. In particular we may define $b'=b-\nabla^{0,1}_Jc$ so that
\[ [a]=[\,\nabla^{0,1}_Jb'\,]\in\oH^p_{\nabla_J^{0,1}}\left(\ker \left(\nabla^{0,1}\colon A^{0,\ast}\to A^{0,\ast+1}\right)\right) \]
and the injectivity of $\mathsf{g}$ follows.

By the strong $\nabla^{0,1}\nabla_J^{0,1}$-lemma, the spectral sequence degenerates at the second page $E_2$, so that the $k$-th cohomology of the total complex is given by the direct sum
\[ \oH^k_{x\nabla^{0,1}_J+y\nabla^{0,1}}(\qA^{\ast}(E)) = \bigoplus_{p+q=k}\,x^py^q\,\oH^k_{y\nabla^{0,1}}(\qA^{\ast}(E)) \; .\]
The same argument taking the first page with respect to the horizontal differential $x\nabla^{0,1}_J$ provides the latter isomorphism in the statement.
\end{proof}

\begin{cor}\label{cor:qD not formal}
The triple $\left(\qA^{\ast,\ast}(E),y\nabla^{0,1},x\nabla^{0,1}_J\right)$ is not a DGMS vector space.
\begin{proof}
If it were a DGMS vector space, then by Lemma~\ref{lemma.DGMScohomology} the associated spectral sequence would degenerate at $E_1$, but we have seen in Lemma~\ref{lemma:cohom of qD} that it degenerates at $E_2$.
\end{proof}
\end{cor}

Recall that on the Dolbeault algebra $A^{0,\ast}(End(E))$ there is an associative product, defined as
\begin{equation}\label{eqn:product on End} 
(\alpha\otimes f)\cdot (\beta\otimes g) = \alpha\wedge\beta\otimes (f\circ g)
\end{equation}
for every $\alpha,\beta\in A^{0,\ast}$ and every sections $f$ and $g$ of $End(E)$. In particular
\[ (A^{0,\ast}(End(E)),\cdot,[\nabla^{0,1},-]) \]
is an associative DG algebra.
%whose corresponding DG Lie algebra controls the holomorphic deformations of $E$, see Subsection~\ref{section:background}.
This of course extends to an associative product on the quaternionic Dolbeault complex $\qA^{\ast}(End(E))$. It is therefore natural to ask if this product endows the quaternionic Dolbeault complex with a structure of associative DG algebra, i.e. if the differential of $\qA^{\ast}(End(E))$ behaves as a derivation.

\begin{rmk}\label{rmk.End(E)connection}
Given a holomorphic vector bundle $E$ and a connection $\nabla$ on it, the connection $\widetilde{\nabla}$ on $End(E)$ induced by $E$  is $\widetilde{\nabla}=[\nabla,-]$. More explicitly, for any section $h$ of $End(E)$ we have 
$$ (\widetilde{\nabla}h)(s)=\nabla h(s)-h(\nabla s).$$
We claim that $\widetilde{\nabla}$ behaves like a derivation with respect to the associative product on $A^{0,\ast}(End(E))$.

In fact, let us take two sections $f$ and $g$ of $End(E)$ and a vector field $\xi$ (i.e.\ $\xi$ is a section of the complexified tangent bundle $T_M$). Then
\begin{itemize}
\item $\widetilde{\nabla}_{\xi}(f\circ g)(s) = \nabla_{\xi}((f\circ g)(s)) - (f\circ g)(\left(\nabla_{\xi}(s)\right)$
\item $\widetilde{\nabla}_{\xi}(f)(g(s)) = \nabla_{\xi}((f\circ g)(s)) - f\left(\nabla_{\xi}(g(s))\right)$
\item $\widetilde{\nabla}_{\xi}(g)(s) =  \nabla_{\xi}(g(s)) - g(\nabla_{\xi}(s))$.
\end{itemize}
and by the arbitrariety of $s$ and $\xi$ we deduce $\widetilde{\nabla}(f\circ g) = \widetilde{\nabla}(f) g + f\widetilde{\nabla}(g)$.
\end{rmk}

\begin{prop}\label{prop.nablaJderivation}
Both differentials $[\nabla^{0,1},-]$ and $[\nabla^{0,1}_J,-]$ act as derivations with respect to the associative product on $A^{0,\ast}(End(E))$. In particular, all of the following are associative DG algebras:
\[ \left(A^{0,\ast}(End(E)),\cdot,[\nabla^{0,1},-]\right)\,, \quad \left(A^{0,\ast}(End(E)),\cdot,[\nabla_J^{0,1},-]\right)\,, \]
\[ \left(\qA^{0,\ast}(End(E)),\cdot,[x\nabla_J^{0,1}+y\nabla^{0,1},-]\right)\; . \]
\begin{proof}
It is sufficient to show that both differentials $[\nabla^{0,1},-]$ and $[\nabla^{0,1}_J,-]$ act as derivations with respect to the product on $A^{0,\ast}(End(E))$. This is quite classical and can be checked directly for $[\nabla^{0,1},-]$, so that the only non-trivial case is for $[\nabla^{0,1}_J,-]$.
Recall that by Remark~\ref{rmk.End(E)connection}, $\widetilde{\nabla}$ acts as a derivation and then so does \[\widetilde{\nabla}^{1,0}=\widetilde{\nabla}-\widetilde{\nabla}^{0,1}=\widetilde{\nabla}\, - \, [\nabla^{0,1},-] \; . \]
Now, since $J$ acts multiplicatively on $A^{0,\ast}$, we obtain that the conjugation
\[\widetilde{\nabla}^{0,1}_J = (J^{-1}\otimes\id)\circ[\nabla^{1,0},-]\circ (J\otimes\id) \]
is a derivation with respect to the product of $A^{0,\ast}(End(E))$.
Therefore, to conclude it is enough to show that $\widetilde{\nabla}^{0,1}_J = [\nabla^{0,1}_J,\,-\,]$, which is straighforward.
\end{proof}
\end{prop}

\textbf{Notation.} Let $E$ be a complex vector bundle. If a Chern connection $\nabla$ on $E$ induces a hyper-holomorphic connection $[\nabla,-]$ on $End(E)$, then $\nabla$ is called \emph{projectively hyper-holomorphic}.

\begin{thm}\label{thm:formality of qD}
Consider a complex vector bundle $E$ endowed with a projectively hyper-holomorphic connection $\nabla$. Then $\left(A^{0,\ast}(End(E)),\nabla^{0,1}\right)$ is a formal associative algebra.
\begin{proof} 
By Lemma~\ref{lemma:debardebarJlemma} and Proposition~\ref{prop.nablaJderivation}, the triple $\left(A^{0,\ast}(End(E)),\nabla^{0,1},\nabla^{0,1}_J\right)$ is an associative DGMS algebra. Therefore the claim follows from Theorem~\ref{thm.DGMS}.
\end{proof}
\end{thm}

As we already remarked in Corollary~\ref{cor:qD not formal}, the quaternionic Dolbeault algebra is not a DGMS algebra. Nevertheless, following the analogy with the case of flat connections studied in~\cite{GM88}, it is natural to ask whether it is formal or not (we will clarify better this analogy in Section~\ref{section:moduli of connections}).

\begin{question}\label{question:qD formal}
Is $\left(\qA^{\ast}(End(E)),[x\nabla^{0,1}_J+y\nabla^{0,1},-]\right)$ a formal associative algebra?
\end{question}

We conclude by highlighting an interesting relation between the classical Dolbeault DG Lie algebra $\left(A^{0,\ast}(End(E)),\nabla^{0,1}\right)$ and the quaternionic Dolbeault DG Lie algebra $\left(\qA^{\ast}(End(E)),[x\nabla^{0,1}_J+y\nabla^{0,1},-]\right)$.

\begin{prop}\label{prop:qD hom abelian}
The classical Dolbeault DG Lie algebra is homotopy abelian if and only if quaternionic Dolbeault DG Lie algebra is homotopy abelian.
\begin{proof}
First suppose $\left(\qA^{\ast}(End(E)),[x\nabla^{0,1}_J+y\nabla^{0,1},-]\right)$ is homotopy abelian. There is an injective morphism
\[ \mathsf{g}\colon \left(A^{0,\ast}(End(E)),\nabla^{0,1}\right) \longrightarrow \left(\qA^{\ast}(End(E)),[x\nabla^{0,1}_J+y\nabla^{0,1},-]\right) \]
given by the multiplication with the central variable $y$.
This map is injective in cohomology by Lemma~\ref{lemma.DGMScohomology}.
Therefore the conclusion follows by the standard homotopy abelianity transfer (see e.g.  \cite{KKP,IM13,Ma15}).

For the converse, define the \emph{extended quaternionic Dolbeault complex}  $\left(\eA^{\ast}(End(E)),[x\nabla^{0,1}_J+y\nabla^{0,1},-]\right)$, as
\[ \eA^{k}(End(E))=\bigoplus_{p+q=k}x^py^qA^{0,k}(End(E))\, \]
where $p,q\in\ZZ$ can be negative.
There is a natural injective morphism of associative algebras
\[ \mathsf{f}\colon\left(\qA^{\ast}(End(E)),[x\nabla^{0,1}_J+y\nabla^{0,1},-]\right)\longrightarrow\left(\eA^{\ast}(End(E)),[x\nabla^{0,1}_J+y\nabla^{0,1},-]\right). \]
Since $\left(\eA^{\ast,\ast}(End(E)),[y\nabla^{0,1},-], [x\nabla^{0,1}_J,-]\right)$ is a DGMS algebra, it is easy to check that 
\[ \oH^k_{[x\nabla^{0,1}_J+y\nabla^{0,1},-]}(\eA^{\ast,\ast}(End(E)))=\bigoplus_{p+q=k}\Ext^k(E,E) \]
(cf.\ Lemma~\ref{lemma.DGMScohomology}). In particular, $\mathsf{f}$ is injective in cohomology. 

Now, if $\left(A^{0,\ast}(End(E)),\nabla^{0,1}\right)$ is homotopy abelian, it is easy to see that also $\eA^{\ast}(End(E))$ is homotopy abelian. In fact it is formal (Proposition~\ref{prop.DGMStrick} and Theorem~\ref{thm.DGMS}) and the associative product is commutative in cohomology, since it is induced by the commutative product on the cohomology of $A^{0,\ast}(End(E))$.
\end{proof}
\end{prop}

\begin{rmk}\label{rmk.extendedqD}
Concerning Proposition~\ref{prop:qD hom abelian}, it is important to point out that the strong $y\nabla^{0,1}\,x\nabla^{0,1}_J$-lemma fails for the quaternionic Dolbeault bi-complex %$\left(\qA^{\ast,\ast}(End(E)),y\nabla^{0,1},x\nabla^{0,1}_J\right)$ 
exactly at the hedge of the bi-complex, i.e.\ either at $\qA^{0,\ast}(End(E))$ or at $\qA^{\ast,0}(End(E))$. This is due to the fact that the indices are supposed to be positive. We can then fix this defect by considering the \emph{extended quaternionic Dolbeault complex} introduced in the proof above. $\left(\eA^{\ast}(End(E)),[x\nabla^{0,1}_J+y\nabla^{0,1},-]\right)$, where 
\[ \eA^{k}(End(E))=\bigoplus_{p+q=k}x^py^qA^{0,k}(End(E)). \]
Notice that here $p,q\in\ZZ$ can be negative.

It is easy to check now that the strong $y\nabla^{0,1}\,x\nabla^{0,1}_J$-lemma holds on $\left(\eA^{\ast,\ast}(End(E)),[y\nabla^{0,1},-], [x\nabla^{0,1}_J,-]\right)$ (just use Lemma~\ref{lemma:debardebarJlemma} and the fact that the variable $x$ and $y$ are central). In particular, the extended quaternionic Dolbeault complex is a formal associative algebra by Proposition~\ref{prop.DGMStrick} and Theorem~\ref{thm.DGMS}.
\end{rmk}

%%%%%%%%%%%%%%%%%%%%%%%%%%%%%%%%%%%%
%%%%%%%%%%%%%%%%%%%%%%%%%%%%%%%%%%%%
%%%%%%%%%%%%%%%%%%%%%%%%%%%%%%%%%%%%
%%%%%%%%%%%%%%%%%%%%%%%%%%%%%%%%%%%%
%%%%%%%%%%%%%%%%%%%%%%%%%%%%%%%%%%%%

\subsection{Deformation functor of the quaternionic Dolbeault DG Lie algebra}

To state the next result, let us recall that every DG Lie algebra $(L,d,[-,-])$ induces a deformation functor $\Def_L\colon \Art_{\CC}\longrightarrow\Sets$
defined by
\[ \Def_L(B) = \frac{\MC(L\otimes_{\CC}\mathfrak{m}_B)}{\sim} \; , \]
where $\mathfrak{m}_B$ is the maximal ideal of $B\in\Art_{\CC}$, see $\S$~\ref{section.DefLie}. We shall adopt the standard notation $\CC[\epsilon]=\CC[t]/(t^2)$ for the so-called algebra of dual numbers. The \emph{tangent space} of the deformation functor $\Def_L$ is by definition
\[ T^1\Def_L = \Def_L(\CC[\epsilon]) \; . \]

If $E$ is a complex vector bundle on $X$ endowed with a projectively hyper-holomorphic connection $\nabla$, we shall denote by $\Def_{\qA^{\ast}(End(E))}$ the deformation functor associated to the quaternionic Dolbeault DG Lie algebra, whose Lie bracket is given by the graded commutator of the associative product, see Proposition~\ref{prop.nablaJderivation}.

\begin{prop}\label{prop:surjectivity}
Let $E$ be a complex vector bundle and $\nabla=\nabla^{1,0}+\nabla^{0,1}$ a projectively hyper-holomorphic connection. Consider the vector bundle $End(E)$ with the induced hyper-holomorphic connection $[\nabla,-]$.
Then the span of DG Lie algebras
\[ \left(A^{0,\ast}(End(E))\,,\,[\nabla^{0,1},-]\right) \stackrel{\hat{\pi}_y}{\longleftarrow} \qA^{\ast}(End(E)) \stackrel{\hat{\pi}_x}{\longrightarrow} \left(A^{0,\ast}(End(E))\,,\,[\nabla^{0,1}_J,-]\right) \]
given by evaluation maps induces natural transformations of deformation functors
\[ \Def_{\left(A^{0,\ast}(End(E))\,,\,[\nabla^{0,1},-]\right)} \stackrel{\pi_y}{\longleftarrow} \Def_{\qA^{\ast}(End(E))} \stackrel{\pi_x}{\longrightarrow} \Def_{\left(A^{0,\ast}(End(E))\,,\,[\nabla^{0,1}_J,-]\right)} \]
that are surjective for every $B\in\Art_{\mathbb{C}}$.

\begin{proof}
First notice that the maps $\pi_x$ and $\pi_y$ are induced by the natural projections of DG Lie algebras
\[ \left(A^{0,\ast}(End(E))\,,\,[\nabla^{0,1},-]\right) \stackrel{\hat{\pi}_y}{\longleftarrow} \qA^{\ast}(End(E)) \stackrel{\hat{\pi}_x}{\longrightarrow} \left(A^{0,\ast}(End(E))\,,\,[\nabla^{0,1}_J,-]\right) \]
where $\pi_x$ (respectively, $\pi_y$) is given by evaluating $x=1,y=0$ (respectively, $x=0,y=1$) in the quaternionic Dolbeault complex.

Now, $[\nabla,-]$ is hyper-holomorphic on $End(E)$ by hypothesis. Therefore, as in the proof of Theorem~\ref{thm.DGMS}, it follows that the natural injection  
\[ \left(\ker([\nabla^{0,1}_J,-])\,,\,[\nabla^{0,1},-]\right)\hookrightarrow \left(A^{0,\ast}(End(E))\,,\,[\nabla^{0,1},-]\right) \]
is a quasi-isomorphism of DG Lie algebras, hence inducing a natural isomorphism
\[ \Def_{\left(\ker(\nabla^{0,1}_J)\,,\,[\nabla^{0,1},-]\right)}\,\cong\, \Def_{\left(A^{0,\ast}(End(E))\,,\,[\nabla^{0,1},-]\right)} \; . \]
This means that in order to prove surjectivity of $\pi_x$ it is sufficient to show the following implication
\[ b_B\in\Def_{\left(\ker([\nabla^{0,1}_J,-])\,,\,[\nabla^{0,1},-]\right)}(B) \qquad \Longrightarrow \qquad y\, b_B\in \Def_{\qA^{\ast}(End(E))}(B) \]
for every $B\in\Art_{\CC}$.
To this aim, fix an element $b_B$ as above and consider
the equalities
\[ \left[x\nabla^{0,1}_J+y\nabla^{0,1}\,,\,y b_B\right]+\frac{1}{2}\left[y b_B\,,\,y b_B\right]=y^2\left(\left[\nabla^{0,1}\,,\,b_B\right]+\frac{1}{2}\left[b_B\,,\,b_B\right]\right)=0 \; , \]
where we used that $b_B$ is in the kernel of $[\nabla^{0,1}_J\,,\,-]$ and that it satisfies the Maurer--Cartan equation in the DG Lie algebra $\left(A^{0,\ast}(End(E))\otimes\mathfrak{m}_B\,,\,[\nabla^{0,1},-]\otimes\id_{\mathfrak{m}_B}\right)$ by assumption. Therefore $y\, b_A$ satisfies the Maurer--Cartan equation in the DG Lie algebra
\[ \left(\qA^{\ast}(End(E))\otimes\mathfrak{m}_B\,,\,[x\nabla^{0,1}_J+y\nabla^{0,1},-]\otimes\id_{\mathfrak{m}_B}\right) \]
as required. A similar argument shows the surjectivity of the map $\pi_y$.
\end{proof}
\end{prop}

\begin{prop}\label{prop.autodualVSdefor}
Let $E$ be a complex vector bundle and $\nabla=\nabla^{1,0}+\nabla^{0,1}$ a projectively hyper-holomorphic connection. Then:
\begin{enumerate}
\item there exists a natural transformation of deformation functors
\[ \Def_{\qA^{\ast}(End(E))}\, \xrightarrow{\;\pi_y\times\,\pi_x\;}\, \Def_{\left(A^{0,\ast}(End(E))\,,\,[\nabla^{0,1},-]\right)} \times \Def_{\left(A^{0,\ast}(End(E))\,,\,[\nabla^{0,1}_J,-]\right)} \]
that is an isomorphism on tangent spaces,
\item there is a natural isomorphism of deformation functors
\[ \theta\colon \Def_{\left(A^{0,\ast}(End(E))\,,\,[\nabla^{0,1},-]\right)}\xrightarrow{\cong}\Def_{\left(A^{0,\ast}(End(E))\,,\,[\nabla^{0,1}_J,-]\right)} \; . \]
\end{enumerate}
\begin{proof}
The existence of the natural transformation of item $(1)$ follows by Proposition~\ref{prop:surjectivity}.
It is an isomorphism on $B=\CC[\epsilon]$ by Lemma~\ref{lemma:cohom of qD}. The last claim is a consequence of the homotopy invariance of deformation functors (cf. $\S$~\ref{section.DefLie}) and of Remark~\ref{rmk:same cohomology}, being
\[\left(A^{0,\ast}(End(E))\,,\,[\nabla^{0,1},-]\right) \quad \mbox{ and } \quad \left(A^{0,\ast}(End(E))\,,\,[\nabla^{0,1}_J,-]\right) \]
quasi-isomorphic DG Lie algebras.
\end{proof}
\end{prop}

%%%%%%%%%%%%%%%%%%%%%%%%%%%%%%%%%%%%%%%%%%
%%%%%%%%%%%%%%%%%%%%%%%%%%%%%%%%%%%%%%%%%%
\section{Deformations of holomorphic vector bundles}\label{section:moduli spaces holomorphic}

%%%%%%%%%%%%%%%%%%%%%%%%%%%%%%%%%%%%%%%
%%%%%%%%%%%%%%%%%%%%%%%%%%%%%%%%%%%%%%%

\subsection{Background on deformations of sheaves}\label{section:background}
To any coherent sheaf $\sF$ on a smooth complex projective manifold $X$ it is associated the homotopy class of DG Lie algebras of derived endomorphisms $\RHom(\sF,\sF)$. In particular such class can be described by different quasi-isomorphic representatives (cf. e.g.~\cite{FIM,Me}) but the associated deformation functor $\Def_{\RHom(\sF,\sF)}$ does not depend on this choice, see $\S$~\ref{section.DefLie}.
It is well known that the deformation functor $\Def_{\sF}$ describing infinitesimal deformations of $\sF$ is naturally isomorphic to $\Def_{\RHom(\sF,\sF)}$, see \cite{AS,BZ,FIM, IM, Me}. This turns out to be a very powerful tool since the geometric properties of certain moduli spaces of sheaves on $X$ can be understood, locally around $[\sF]$, by the algebraic properties of $\RHom(\sF,\sF)$. For instance, it has been proven that for any polystable coherent sheaf $\sF$ on $X$ the class of DG Lie algebras $\RHom(\sF,\sF)$ is formal when $X$ is a $K3$ surface (cf.~\cite{BZ}) and even more generally whenever $X$ is a minimal surface of Kodaira dimension $0$ (cf.~\cite{BaMaMe21}). 

For future reference we state here a very useful technical tool, whose proof can be found either in \cite{FMM} or in \cite{FIM}. 

\begin{prop}\label{prop:deformations with dgL}
Let $X$ be a smooth and projective variety and let $(E,\debar_E)$ be a holomorphic vector bundle on $X$. Then the Dolbeault algebra $(A^{0,\ast}(End(E)),[\debar_E,-],[-,-])$ represents the homotopy class $\,\RHom(E,E)$.
\end{prop}

As a consequence of $\S$~\ref{section.DefLie} and Proposition~\ref{prop:deformations with dgL} one has the following geometric interpretation, which nowadays is folklore, and for which a good account can be found in the seminal paper \cite{GM90}.
Let us denote by $\Def(E,\debar_E)$ the Kuranishi family of a holomorphic vector bundle $(E,\debar_E)$. 
%Denote by $\Def_E\cong\Def_{\RHom(E,E)}$ the deformation functor of infinitesimal deformations of $E$.

\begin{cor}\label{cor.formalityimpliesquadraticity}
Keep notations as before.
\begin{enumerate}
\item If $\,(A^{0,\ast}(End(E)),[\debar_E,-],[-,-])$ is Lie formal, then $\Def(E,\debar_E)$ is an intersection of quadrics.
\item If $\,(A^{0,\ast}(End(E)),[\debar_E,-],[-,-])$ is homotopy abelian, then $\Def(E,\debar_E)$ is smooth.
\end{enumerate}
\end{cor}

With the above result in mind, the aim of $\S$~\ref{section.formalityhyperholomorphic} will be to provide sufficient conditions for
\[ \left(A^{0,\ast}(End(E)),[\debar_E,-],[-,-]\right) \]
to be associatively (hence Lie) formal.
Notice that the implications of Corollary~\ref{cor.formalityimpliesquadraticity} are strong, indeed a priori there is no reason to expect the above conditions to be equivalent. In particular homotopy abelianity is stronger than smoothness. Nevertheless, if $(E,\debar_E)$ is polystable with respect to any polarisation and $X$ is a smooth projective surface it has been proved in \cite{BaMaMe22} that formality is equivalent to quadraticity (in fact, it has been proven a more general statement for coherent polystable sheaves satisfying certain conditions on a complex projective scheme of arbitrary dimension, cf.~\cite[Theorem 1.2]{BaMaMe22}).

%\subsection{Local behaviour of projectively hyper-holomorphic vector bundles}
\subsection{Formality results for vector bundles with a projectively hyper-holomorphic connections}\label{section.formalityhyperholomorphic}

Let $X=(M,I)$ be an irreducible holomorphic symplectic manifold with K\"ahler class $\omega$, and $(M,g,I,J,K)$ the associated hyper-K\"ahler structure.
%Recall that here $E$ is the comnplex vector bundle and $\debar_E$ is the fixed holomorphic structure.
Let $E$ be a complex vector bundle endowed with a hyper-holomorphic connection $\nabla=\nabla^{1,0}+\nabla^{0,1}$. Recall that in particular $\nabla^{0,1}$ is a holomorphic structure on $E$.
As in Section $\S$~\ref{section:def of hyper-hol connections} (cf.Remark~\ref{rmk.qDdoublecomplex}) we shall denote by \[ \nabla^{0,1}_J\colon A^{0,\ast}(E)\rightarrow A^{0,\ast+1}(E) \;  \]
the associated twisted differential.

Motivated by the deformation theory of a holomorphic vector bundle $(E,\nabla^{0,1})$, we now look at its endomorphism bundle.
%Then the hermitian metric of $E$ is called \emph{projectively hyper-holomorphic} if the induced hermitian metric on $End(E)$ is hyper-holomorphic.
Recall that the Chern connection $\nabla$ on an hermitian vector bundle $E$ is called projectively hyper-holomorphic if the induced connection on $End(E)$ is hyper-holomorphic.

\begin{thm}\label{thm:formality of D}\label{thm:End are DGMS}
Consider a complex vector bundle $E$ endowed with a projectively hyper-holomorphic connection $\nabla$. Then:
\begin{enumerate}
    \item $\left(A^{0,\ast}(End(E))\,,\,\cdot\,,\,[\nabla^{0,1}_J,-]\,,\,[\nabla^{0,1},-]\right)$ is an associative DGMS algebra,
    \item $\left( A^{0,\ast}(End(E))\,,\,\cdot\,,\, [\nabla^{0,1},-] \right)$ is a formal associative DG algebra.
\end{enumerate}
\begin{proof}
By Lemma~\ref{lemma:debardebarJlemma} the Dolbeault complex is a DGMS vector space. Hence it is sufficient to show that the differentials act as derivations with respect to the natural associative product of $A^{0,\ast}(End(E))$, which follows by Proposition~\ref{prop.nablaJderivation}. The last statement is then an immediate consequence of Theorem~\ref{thm.DGMS}.
\end{proof}
\end{thm}

\begin{prop}[\protect{\cite[Theorem~11.1]{Verbitsky:JAG1996}}]\label{prop:verbitsky Delta}
Let $(E,\debar_E)$ be a stable holomorphic vector bundle on $X$ and consider the hermitian Yang--Mills metric induced by the Donaldson--Yau--Uhlenbeck theorem. Then the Chern connection on $E$ is projectively hyper-holomorphic if and only if 
$$\Delta(E)\defeq 2rc_2(E)-(r-1)c_1(E)^2$$
is $\SU(2)$-invariant.
\end{prop}

\begin{example}
The Chern connection of any line bundle is projectively hyper-holomorphic.
\end{example}

\begin{example}\label{example:K3}
Let $X$ a K3 surface, and $E$ be a stable holomorphic vector bundle. The Yang--Mills Chern connection on $E$ provided by Donaldson--Uhlenbeck--Yau theorem is projectively hyper-holomorphic. Essentially this is due to the fact that in dimension $2$ any $2$-form $\alpha$ such that $\Lambda_\omega\alpha=0$ is $\SU(2)$-invariant, see~\cite[Proposition~11.2]{Verbitsky:JAG1996}.
\end{example}

\begin{rmk}
Hermitian holomorphic vector bundles on $X$ whose Chern connection is projectively hyper-holomorphic are examples of modular vector bundles, see~\cite[Definition~1.1]{O'Grady:Modular}. As already remarked by O'Grady~\cite[Remark~1.3]{O'Grady:Modular}, the two definitions agree when $X$ is deformation equivalent to a Hilbert scheme of $2$ points on a K3 surface, or when $X$ is deformation equivalent to the so-called OG10 manifold. Nevertheless it is not known in general whether this is always the case or not: in particular, there are no examples of stable modular vector bundles that do not admit projectively hyper-holomorphic connections.
\end{rmk}

\begin{thm}\label{thm.Dolbeaultformal}
Let $X$ be an irreducible holomorphic symplectic manifold $X$ and $\omega$ a K\"ahler class on $X$. If $(E,\debar_E)$ is a stable holomorphic vector bundle on $X$ such that $\Delta(E)$ is $\SU(2)$-invariant, then the associative Dolbeault DG algebra $\left(A^{0,\ast}(End(E))\,,\,\cdot\,,\,[\debar_E,-]\right)$ is formal.
\begin{proof}
By Proposition~\ref{prop:verbitsky Delta}, the hypotheses imply that $E$ admits a projectively hyper-holomorphic connection. Then the associative formality of the Dolbeault DG algebra follows by Theorem~\ref{thm:End are DGMS}.
\end{proof}
\end{thm}

By Remark~\ref{rmk:same cohomology} we get that $\oH^{\bullet}_{[\debar_E,-]}\left(A^{0,\ast}(End(E))\right) \cong \oH^{\bullet}_{[\nabla^{0,1}_J,-]}\left(A^{0,\ast}(End(E))\right)$ . This should be compared to \cite[Corollary~8.2]{Verbitsky:JAG1996}, where it is shown that the vector spaces $\oH^k(End(E))$ do not depend on the induced complex structure, when $E$ is a projectively hyper-holomorphic hermitian vector bundle.

\begin{rmk}[Polystable vector bundles]
If $E=\bigoplus_{k=1}^mE_k$ is a polystable vector bundle, i.e.\ each $E_k$ is stable, then the Donaldson--Uhlenbeck--Yau Yang--Mills Chern connection $\nabla=\sum_{k=1}^m\nabla_k$ can be projectively hyper-holomorphic even if the single connections $(E_k,\nabla_k)$ are not. In this case Theorem~\ref{thm:End are DGMS} still applies, but Proposition~\ref{prop:verbitsky Delta} does not a priori. In fact the $\SU(2)$-invariance of $\Delta(E)$ does not imply the $\SU(2)$-invariance of each $\Delta(E_k)$ and we cannot conclude that the Donaldson--Uhlenbeck--Yau Yang--Mills Chern connection $\nabla$ is projectively hyper-holomorphic. In fact, for this to be true, one has to further assume that each summand $\nabla_k$ is so.
\end{rmk}

In the following we state some geometric consequences of Theorem~\ref{thm.Dolbeaultformal}. If $(E,\debar_E)$ is a holomorphic vector bundle, we denote by $\Def(E,\debar_E)$ the Kuranishi space of infinitesimal holomorphic deformations of $E$ (cf.\ \cite{GM90}).

The first consequence is a shorter proof of a result due to Verbitsky, ~\cite[Theorem~11.2]{Verbitsky:JAG1996}.

\begin{cor}[Verbitsky's quadraticity theorem]\label{cor.formalityVSquadraticity}
Let $X$ be an irreducible holomorphic symplectic manifold $X$ and $\omega$ a K\"ahler class on $X$. If $(E,\debar_E)$ is a stable holomorphic vector bundle that admits a projectively hyper-holomorphic connection, then $\Def(E,\debar_E)$ is an intersection of quadrics.
\begin{proof}
This is a direct consequence of the formality of Theorem~\ref{thm:End are DGMS} and Corollary~\ref{cor.formalityimpliesquadraticity}.
\end{proof}
\end{cor}

\begin{rmk}
Verbitsky's proof of quadraticity of $\Def(E,\debar_E)$ is very technical and heavily uses the differential and Hodge theory of $A^{0,\ast}(End(E))$. Nevertheless his main tool can be tracked back to the $\nabla^{0,1}\,\nabla^{0,1}_J$-lemma, which is also our main tool to prove formality. %For this reason we are led to expect that, similarly to the case of K3 surfaces, quadraticity and formality are equivalent to one another also in the case of polystable projectively hyper-holomorphic vector bundles on irreducible holomorphic symplectic manifolds. 
%From this point of view formality and quadraticity can be considered as ``the same" property of stable projectively hyper-holomorphic vector bundles (cf.\ \cite{BaMaMe22} for the analogous statement for sheaves on K3 surfaces).
%Recall that for smooth projective surfaces, quadraticity is equivalent to formality as soon as the automorphism group of the coherent sheaf is reductive, see~\cite{BaMaMe22}. It would be interesting to understand whether this equivalence can be further proved for smooth projective higher dimensional varieties under some hypothesis.
\end{rmk}

Another consequence is the following well known result (see~\cite{Kim}).

\begin{cor}\label{cor.Kim}
Let $X$ be an irreducible holomorphic symplectic manifold $X$ and $\omega$ a K\"ahler class on $X$. Let $(E,\debar_E)$ be a stable holomorphic vector bundle that admits a projectively hyper-holomorphic connection. If moreover $\Ext^2(E,E)=\CC$, then $\Def(E,\debar_E)$ is smooth.
\begin{proof}
First of all,  by Theorem~\ref{thm:End are DGMS} the algebra $\left(A^{0,\ast}(End(E))\,,\,\cdot\,,\,[\debar_E,-]\right)$ is formal. Moreover, recall that $\oH^{\bullet}(A^{0,\ast}(End(E))=\Ext^{\bullet}(E,E)$. By assumption $\Ext^2(E,E)=\CC$, so that in particular the cap product defined on $\Ext^1(E,E)=\oH_{\debar}^1(A^{0,\ast}(End(E)))$ is skew-symmetric and $\Def(E,\debar_E)$ is smooth (cf. $\S$~\ref{section.DefLie}).
\end{proof}
\end{cor}

Notice that in Corollary~\ref{cor.Kim}, the formal DG Lie algebra $A^{0,\ast}(End(E))$ need not to be homotopy abelian in general.%; in fact the skew-symmetry of the Yoneda product $\Ext^1(E,E)\times\Ext^1(E,E)\to\Ext^2(E,E)$ together with formality is sufficient to obtain the smoothness of $\Def(E,\debar_E)$, cf. $\S$~\ref{section.DefLie}.

\begin{rmk}[Moduli space of semistable vector bundles on K3 surfaces]
We already noticed in Example~\ref{example:K3} that every polystable vector bundle on a K3 surface admits a hyper-holomorphic connection. Moreover, if $(E,\debar_E)$ is stable then $\Ext^2(E,E)\cong\Hom(E,E)=\CC$.
Hence we recover the well known fact that the moduli space of semistable holomorphic vector bundles on a K3 surface is quadratic \cite{BZ,BaMaMe22}, and it is smooth at those points that can be represented by a stable holomorphic vector bundle.
\end{rmk}

\begin{cor}\label{cor:3}
Let $X$ be an irreducible holomorphic symplectic manifold $X$ and $\omega$ a K\"ahler class on $X$. Let $(E,\debar_E)$ be a stable holomorphic vector bundle that admits a projectively hyper-holomorphic connection. If moreover there exists a smooth projective variety $Z$ such that $\Ext^{\ast}(E,E)\cong \oH^{\ast}(Z,\CC)$, then
\[ \left(A^{0,\ast}(End(E))\,,\,[\debar_E,-]\,,\,[-,-]\right) \]
is a homotopy abelian DG Lie algebra. In particular $\Def(E,\debar_E)$ is smooth.
\begin{proof}
By Theorem~\ref{thm.Dolbeaultformal}, the associative Dolbeault DG algebra $\left(A^{0,\ast}(End(E))\,,\,\cdot\,,\,[\debar_E,-]\right)$ is formal. Hence to prove homotopy abelianity is enough to show that the associative product is graded commutative in cohomology, see Remark~\ref{rmk.formalityVShomotopyabelianity}. On the other hand, the associative cap product on $\oH^{\ast}(Z,\CC)$ is graded commutative so that the claim follows by a standard Eckmann-Hilton argument,~\cite{EH}.
\end{proof}
\end{cor}

\begin{cor}[\protect{\cite[Proposition~1.4]{BZ}}]\label{cor:BZ}
Let $X$ and $Y$ be two projective irreducible holomorphic symplectic manifolds and $\Phi\colon\operatorname{D}^b(X)\to\operatorname{D}^b(Y)$ a Fourier--Mukai equivalence between their derived categories. Suppose that $(E,\debar_E)$ is a hermitian holomorphic vector bundle admitting a projectively hyper-holomorphic connection. Then $\Def(\Phi(E),\debar_{\Phi(E)})$ is given by an intersection of quadrics.
\end{cor}
\begin{proof}
By Theorem~\ref{thm.Dolbeaultformal} the Dolbeault associative DG algebra $\left(A^{0,\ast}(End(E))\,,\,\cdot\,,\,[\debar_E,-]\right)$ is formal. On the other hand, by \cite[Proposition~1.4]{BZ} the associative formality of derived endomorphisms is preserved under Fourier--Mukai equivalences. The statement follows.
\end{proof}

\begin{example}[Tangent bundle]
In Example~\ref{example:T} we have seen that the Chern connection of the holomorphic tangent bundle $\Theta_X$ of an irreducible holomorhic symplectic manifold $X$ is hyper-holomorphic with respect to the hermitian metric induced by $X$. Then the Dolbeault algebra $A^{0,\ast}(End(\Theta_X))$ is formal and the Kuranishi space $\Def(\Theta_X)$ is quadratic. 
\end{example}

\begin{example}[Vector bundles on Fano varieties of lines]
This example is due to O'Grady  (see \cite[Section~2.1]{O'Grady:Modular}). Let $Y\subset\lP(V_6)$ be a smooth cubic fourfold and $X=F(Y)\subset G=G(2,V_6)$ its Fano variety of lines. It is known that $X$ is an irreducible holomorphic symplectic manifold of type $K3^{[2]}$ (\cite{BD}). The Pl\"ucker embedding of $G$ in $\lP(\bigwedge^2V_6)$ provides $X$ with a canonical polarisation $H$ and, if $X$ is very general, it follows that $\Pic(X)=\ZZ H$. Let $Q$ be the tautological bundle on $G$ parametrising the quotient spaces of $V_6$; then $Q$ has rank $4$ and first Chern class equal to the Pl\"ucker class. One can check (e.g.\ Example~(2) in \cite[Section~2.1]{O'Grady:Modular}) that $\Delta(E)=c_2(X)$ is $\SU(2)$-invariant. Moreover, $E$ is stable. 
In particular, $E$ admits a projectively hyper-holomorphic connection and $\Def(E)$ is quadratic. 

Since $Q$ is generated by global sections, one can check that also $E$ is generated by global sections (in particular one can compute that $\oH^0(X,E)\cong \oH^0(G,Q)=V_6$ -- we thank E.\ Fatighenti for this remark). The elementary transform of $E$ is a vector bundle $F$ on $X$ of rank $2$ that is isomorphic to the restriction to $X$ of the tautological bundle on $G$ of the subspaces of $V_6$. It is easy to check by hand that $F$ is stable and that $\Delta(F)=\frac{3}{2}h^2-\frac{1}{2}c_2(X)$ is not $\SU(2)$-invariant. 
Nevertheless, by Corollary~\ref{cor:BZ}, also $A^{0,\ast}(End(F))$ if formal, and hence $\Def(F)$ is quadratic. 
From a geometric point of view these examples are easily understood: both $E$ and $F$ are rigid vector bundles, so that their deformation spaces consist of just a point.
\end{example}

\begin{example}
In \cite{O'Grady22}, O'Grady constructs a class of slope stable locally free sheaves on a generalised Kummer fourfold whose discriminant is $\SU(2)$-invariant. Therefore, by Theorem~\ref{thm.Dolbeaultformal}, the associated Dolbeault algebra is formal. Notice that also in this case the locally free sheaf is rigid.
\end{example}

\subsubsection{Other examples}
Constructing examples of projectively hyper-holomorphic connections on holomorphic vector bundles is in general a very difficult task. In fact, the only tool we have now is Proposition~\ref{prop:verbitsky Delta} applied to stable holomorphic vector bundles. The biggest set of theoretical examples is provided by Markman in \cite{Markman:Modular}. Markman's examples are divided in three classes: 
\begin{enumerate}
\item vector bundles of the form $\Phi(\cO_X)$, where $\Phi$ is a Fourier--Mukai transform;
\item vector bundles of the form $\Phi(\cO_x)$, where $\Phi$ is as before and $x\in X$ is a point;
\item vector bundles of the form $\Phi(\mathcal{L})$, where $\Phi\colon\operatorname{D}^b(Z)\to\operatorname{D}^b(X)$ a Fourier--Mukai equivalence, $Z\subset X$ is a lagrangian submanifold that deforms with $X$ in codimension $1$ and $\mathcal{L}$ is a power of the anti-canonical bundle of $Z$.
\end{enumerate}

All these examples, when stable, have a smooth Kuranishi space (actually, their Dolbeault DG Lie algebra is homotopy abelian), as it follows from \cite[Proposition~1.4]{BZ} (see also Corollary~\ref{cor:BZ}).  

The first two classes of examples are rigid, so that from a geometric point of view they are not very interesting. Markman explicitly constructs stable vector bundles belonging to these two classes of examples.

The third class of examples is much more interesting, because would provide examples of non-rigid vector bundles. In \cite[Example~3.11]{Markman:Modular}, Markman lists three known geometric cases where a lagrangian submanifold $Z\subset X$ deforms with $X$ in codimension $1$. The most interesting geometric example is given by sheaves of the form $\Phi(i_{\ast}\cO_{F(Y)})$, where $V$ is a smooth cubic fourfold, $Y$ is a linear section of $V$, $F(Y)$ and $F(V)$ are the respective Fano varieties of lines, and $i\colon F(Y)\to F(V)$. Unfortunately, it is still unknown whether such a sheaf is locally free and whether it is stable.
Nevertheless, when stable and locally free such a sheaf has a homotopy abelian Dolbeault DG Lie algebra by Theorem~\ref{thm.Dolbeaultformal} and Corollary~\ref{cor:3}. We remark that the formality of torsion sheaves of the form $i_{\ast}\mathcal{L}$ as above has been proved by Mladenov in \cite{Mladenov:Formality}, see also \cite[Theorem~0.1.3]{Mladenov:Degeneration}.
%Another class of stable sheaves on irreducible holomorphic symplectic manifolds whose derived homomorphism algebra is formal is provided by Beckman in \cite[]{Beckman}. These sheaves are called atomic in loc.\ cit.\ and when locally free they admit a hyper-holomorphic connection by .

%%%%%%%%%%%%%%%%%%%%%%%%%%%%%%%%%%%%%%%%%%%%%%%%%%%%%%%%%
%%%%%%%%%%%%%%%%%%%%%%%%%%%%%%%%%%%%%%%%%
%%%%%%%%%%%%%%%%%%%%%%%%%%%%%%%%%%%%%%%%%
\section{Deformations of autodual connections}\label{section:moduli of connections}
%\section{Local theory of moduli spaces of autodual connections}\label{section:moduli of connections}

%This section is meant to be a contribution to the results developed in \cite{Kaledin.Verbitsky98} about moduli spaces of autodual connections on a complex vector bundle $E$ on an irreducible holomorphic symplectic manifold $X$.

To begin with, our aim is to define the deformation functor associated to an autodual connection $\nabla$ on a complex vector bundle on an irreducible holomorphic symplectic manifold.

First of all, if $B\in\Art_\CC$ is any local artinian $\CC$-algebra with residue field $\CC$, we define the $\SU(2)$-action on $A^\ast(End(E))\otimes_{\CC}B$ as the natural $\SU(2)$-action on $A^\ast(End(E))$ and as the identity on $B$.

\begin{defn}[Infinitesimal deformation of an autodual connection]\label{def.deformationconnection}
Let $X=(M,I)$ be an irreducible holomorphic symplectic manifold and $(M,g,I,J,K)$ the hyper-K\"ahler structure associated to a K\"ahler class $\omega$. Let $E$ be a complex vector bundle on $X$ and $\nabla$ an autodual connection on $E$. A deformation of $\nabla$ over $B\in\Art_{\CC}$ is a $B$-linear graded map
\[ \nabla_B\colon A^{\ast}(E)\otimes_{\CC}B\to A^{\ast+1}(E)\otimes_{\CC}B \]
satisfying the following conditions:
\begin{enumerate}
 \item $\nabla_B$ satisfies the Leibniz rule
 \[ \nabla_B(\alpha\otimes e\otimes b)=(\operatorname{d}\alpha)\otimes e\otimes b+(-1)^{k}\alpha\wedge\nabla_B(e\otimes b), \]
 where $\alpha$ is a $k$-form, $e$ is a section of $E$ and $b\in B$.
 \item $\nabla_B\otimes_B\CC = \nabla$,
 \item $\nabla_B^2\in A^2(End(E))\otimes_{\CC}B$ is $\SU(2)$-invariant.
\end{enumerate}
\end{defn}

We refer to a $\nabla_B$ satisfying item $(1)$ of Definition~\ref{def.deformationconnection} as a $B$-connection.
Item $(2)$ can be understood as the commutativity of the diagram of $B$-modules
\[ \xymatrix{ A^{\ast}(E)\otimes_{\CC} B \ar@{->}[rr]^{\nabla_B} \ar@{->}[d] & & A^{\ast+1}(E)\otimes_{\CC} B  \ar@{->}[d] \\
A^{\ast}(E) \ar@{->}[rr]^{\nabla} & & A^{\ast+1}(E)
} \]
where the $B$-module structure of $A^{0,\ast}(E)$ is induced by the projection $B\to \nicefrac{\mbox{$B$}}{\mbox{$\mathfrak{m}_B$}}\cong\CC$ that annihilates the maximal ideal $\mathfrak{m}_B$.
Finally, notice that the composition $\nabla_B^2=\nabla_B\circ\nabla_B$ is $A^0\otimes_{\CC}B$-linear, so that item $(3)$ is well defined and corresponds to the autoduality property.

\begin{rmk}
We may prove a completely analogous statement to Proposition~\ref{prop:hyper-hol as MC} simply by working over $B$ instead of $\CC$ in order to characterise autodual deformations in terms of strong Maurer--Cartan solutions.
\end{rmk}

\begin{rmk}\label{rmk.connectiondeformations}
Notice that a $B$-deformation $\nabla_B$ of $\nabla$ is equivalent to the data of an $A^0$-linear
\[ \hat{\nabla}_B\colon A^{0}(E)\to A^{1}(E)\otimes_{\CC}\mathfrak{m}_B \; . \]
In fact, if such $\hat{\nabla}_B$ is given, then there is a unique way to recover the $B$-deformation $\nabla_B$ by extending
\[ \nabla\otimes1\,+\,\hat{\nabla}_B\colon A^{0}(E)\to A^{1}(E)\otimes_{\CC}B \; . \]
via $B$-linearity and the Leibniz rule.
\end{rmk}

\begin{defn}[Equivalence of deformations]\label{defn:new gauge}
If $E$ is a complex vector bundle over $X$ and $B\in\Art_\CC$, then we denote by $E\otimes_\CC B$ the vector bundle over $X_B=X\times\operatorname{Spec}(B)$. If $\mathcal{G}(E)$ is the gauge group of $E$ and $\mathcal{G}(E)_B$ is the gauge group of $E\otimes B$, then we define $\mathcal{G}(E)^0$ the kernel
\[ \mathcal{G}(E)_B^0=\ker\left(\mathcal{G}(E)_B\longrightarrow\mathcal{G}(E)\right) \]
induced by the quotient map $B\to \nicefrac{\mbox{$B$}}{\mbox{$\mathfrak{m}_B$}}\cong\CC$.
%by the maximal ideal $\mathfrak{m}_B$ of $B$.

If $\nabla_{B,1}$ and $\nabla_{B,2}$ are two $B$-connections, then we say that they are \emph{equivalent} if there exists an element $g_B\in\mathcal{G}(E)_B^0$ such that the following diagram
\[ \xymatrix{ A^{\ast}(E)\otimes_{\CC}B \ar@{->}[rr]^{g_B}\ar@{->}[d]_{\nabla_{B,1}} & & A^{\ast}(E)\otimes_{\CC}B \ar@{->}[d]^{\nabla_{B,2}} \\
A^{\ast+1}(E)\otimes_{\CC}B \ar@{->}[rr]_{g_B} & & A^{\ast+1}(E)\otimes_{\CC}B }\;  \]
commutes.
\end{defn}

In the diagram above, the automorphism $g_B$ is extended to the complex $A^\ast(E)\otimes B$ by $A^0$-linearity.
%Notice also that by hypothesis we have $g_B\otimes_B\CC=\id_{A^\ast(E)}$. 

\begin{rmk}
The definition above is classical, but let us explain it. First of all, recall that the gauge group $\mathcal{G}(E)$ is the group of transformations of $E$ as a vector bundle over $X$ (same for $\mathcal{G}(E)_B$). Classically two connections (resp.\ $B$-connections) over $E$ (resp.\ $E_B$) are said to be equivalent if they are conjugated by a gauge transformation. The choice to act by the subgroup $\mathcal{G}(E)_B^0$ is motivated by the fact that we want an equivalence between two deformations of $\nabla$ to be the identity when quotiented by the maximal ideal (i.e.\ when restricted to the central fibre).
\end{rmk}
\begin{rmk}\label{rmk:gauge and exp}
The group $\mathcal{G}(E)_B^0$ is classically known in deformation theory, and it coincides with the nilpotent Lie group $\operatorname{exp}(A^0(End(E))\otimes_\CC\mathfrak{m}_B)$ (see e.g.\ \cite[Section~6.1]{GM88}).

In particular, from our point of view it is important to remark that $\operatorname{exp}(A^0(End(E))\otimes_\CC\mathfrak{m}_B)$ is also the group of gauge transformations of the quaternionic Dolbeault DG Lie algebra $\qA^*(End(E))\otimes_\CC B$ (cf.\ Section~$\S$~\ref{section.DefLie}).
\end{rmk}

\begin{defn}[Deformation functor associated to an autodual connection]
Let $X$ be an irreducible holomorphic symplectic manifold with a fixed K\"ahler class $\omega$. Let $E$ be a complex vector bundle on $X$ endowed with an autodual connection $\nabla$.
The deformation functor associated to $\nabla$ is defined by
\[ \Def_{\nabla}\colon \Art_{\CC}\to\Sets \qquad \qquad  \Def_{\nabla}(B) = \nicefrac{\left\{\mbox{autodual deformations of $\nabla$ over $B$}\right\}}{\sim} \; , \]
where $\sim$ denotes the equivalence relation given in Definition~\ref{defn:new gauge}.
\end{defn}

%Now, suppose that $X$ is an irreducible holomorphic symplectic manifold with a fixed K\"ahler class $\omega$. Let $E$ be a complex vector bundle on $X$. %We denote by $\M^s(E,\omega)$ the open subset of $\M^s(E)$ consisting of autodual connections on $E$ with respect to the hyper-K\"ahler structure associated to $\omega$. Our next result describes the local structure of such moduli space.

Given an autodual connection $\nabla$ on $E$, we denote by $\Def_{\qA^{\ast}(End(E))}$ the deformation functor associated to the quaternionic Dolbeault DG Lie algebra \[ \left( \qA^{\ast}(End(E))\,,\, x\,[\nabla^{0,1}_J,-]+y\, [\nabla^{0,1},-] \,,\,[-,-] \right) \]
as in Section $\S$~\ref{section:def of hyper-hol connections}.

The main result of this section is the following.

\begin{thm}\label{thm:iso of def}
Let $X$ be an irreducible holomorphic symplectic manifold with a fixed K\"ahler class $\omega$. Let $E$ be a complex vector bundle on $X$.
If $\;\nabla$ is an autodual connection, then there is a natural isomorphism of deformation functors
\[ \mathsf{\Phi}\colon \Def_{\qA^{\ast}(End(E))}\to \Def_{\nabla}.\]
\begin{proof}
Fix $B\in\Art_{\CC}$. A Maurer--Cartan element
\[ x\,\xi_1\otimes b_1+y\,\xi_2\otimes b_2\in\MC\left( \qA^{\ast}(End(E))\otimes_{\CC}\mathfrak{m}_B\,,\, \left[x\,\nabla^{0,1}_J+y\,\nabla^{0,1},-\right]\otimes\id_{\mathfrak{m}_B} \right) \]
satisfies the equation
\[ \begin{aligned}
&x^2\,\left(\left[\nabla^{0,1}_J\,,\,\xi_1\right]\otimes b_1+\frac{1}{2}\left[\xi_1,\xi_1\right]\otimes b_1^2 \right)+y^2\,\left(\left[\nabla^{0,1}\,,\,\xi_2\right]\otimes b_2+\frac{1}{2}\left[\xi_2,\xi_2\right]\otimes b_2^2 \right)+\\ 
&+xy\,\left( \left[\nabla^{0,1}_J\,,\,\xi_2\right]\otimes b_2 + \left[\nabla^{0,1}\,,\,\xi_1\right]\otimes b_1 + \left[\xi_1,\xi_2\right]\otimes b_1b_2 \right) = 0 \; . \end{aligned} \]
This can be rephrased as the following conditions
\begin{itemize}
    \item $\xi_2\otimes b_2\in\MC\left(A^{0,\ast}(End(E))\otimes_{\CC}\mathfrak{m}_B\,,\,[\nabla^{0,1},-]\otimes\id_{\mathfrak{m}_B}\right)$,
    \item $\xi_1\otimes b_1\in\MC\left(A^{0,\ast}(End(E))\otimes_{\CC}\mathfrak{m}_B\,,\,[\nabla^{0,1}_J,-]\otimes\id_{\mathfrak{m}_B}\right)$,
    \item $\left[\nabla^{0,1}_J +\xi_1\otimes b_1  \,,\, \nabla^{0,1}+\xi_2\otimes b_2\right]=0$.
\end{itemize}
where in the last equality we used $\left[\nabla^{0,1}_J \,,\, \nabla^{0,1}\right]=0$, being $\nabla$ autodual by hypothesis. Once again, we can repackage the above information as follows
\begin{itemize}
    \item $\nabla^{0,1}+\xi_2\otimes b_2$ is a holomorphic deformation of $\nabla^{0,1}$,
    \item $\nabla^{0,1}_J+\xi_1\otimes b_1$ satisfies the \emph{strong} Maurer--Cartan equation in the DG Lie algebra 
    \[ \left(A^{0,\ast}(End(E))\otimes_{\CC}\mathfrak{m}_B\,,\,[\nabla^{0,1}+\xi_2\otimes b_2,-]\otimes\id_{\mathfrak{m}_B}\right) \; . \]
\end{itemize}
Hence it is uniquely defined the $A^0$-linear map
\[ \hat{\nabla}_B = \xi_2\otimes b_2 + (J\otimes\id)\circ(\xi_1\otimes b_1)\circ(J^{-1}\otimes\id) \colon A^0(E)\longrightarrow A^1(E)\otimes_{\CC}\mathfrak{m}_B \]
which in turn gives us the definition of the $B$-deformation $\nabla_B=\nabla+\hat{\nabla}_B$ as explained by Remark~\ref{rmk.connectiondeformations}. Notice that the items above are equivalent to require that $\nabla_B$ is autodual, see Proposition~\ref{prop:hyper-hol as MC}.

We defined a natural transformation from the Maurer--Cartan functor, i.e.\ 
\[ \hat{\mathsf{\Phi}}\colon\MC_{\qA^\ast(End(E))}\longrightarrow\Def_\nabla. \]
It is not difficult to check that such a natural transformation is surjective.

To conclude we need to understand how the natural transformation $f$ behaves with respect to the respective gauge transformations.
First of all, recall from the definition that $\qA^0(End(E))=A^0(End(E))$ and so the gauge group acting on $\MC(\qA^\ast(End(E))\otimes_{\CC}\mathfrak{m}_B)$ coincides with the group $\operatorname{exp}(\qA^0(End(E))\otimes\mathfrak{m}_B)$. On the other hand, the last group is the same as the group of gauge equivalences of a $B$-deformation of $\nabla$ as defined in Definition~\ref{defn:new gauge}. 
To conclude the proof, we claim that 
\[ \hat{\mathsf{\Phi}}\left(x\,\xi_1\otimes b_1+y\,\xi_2\otimes b_2\right)\sim\hat{\mathsf{\Phi}}\left( x\,\xi'_1\otimes b'_1+y\,\xi'_2\otimes b'_2 \right) \]
if and only if
\[ (x\,\xi_1\otimes b_1+y\,\xi_2\otimes b_2)\sim(x\,\xi'_1\otimes b'_1+y\,\xi'_2\otimes b'_2)\;. \]
From this it follows both that $\hat{\mathsf{\Phi}}$ descends to a natural transformation 
\[ \mathsf{\Phi}\colon\Def_{\qA^\ast(End(E))}\longrightarrow\Def_\nabla \]
and that $\mathsf{\Phi}$ is injective, hence an isomorphism.

On the other hand the claim holds by definition. In fact Definition~\ref{defn:new gauge} is given so that it coincides with the natural gauge action in the deformation theory of a DG Lie algebra (see Remark~\ref{rmk:gauge and exp}).
\end{proof}
\end{thm}

If $E$ is a complex vector bundle and $\nabla$ is a hyper-holomorphic connection, then $(E,\nabla^{0,1})$ is a holomorphic structure (that is polystable by the Donaldson--Uhlenbeck--Yau Theorem). From this point of view it is natural to ask what is the interplay between holomorphic deformations of $(E,\nabla^{0,1})$ and autodual deformations of $(E,\nabla)$.

\begin{cor}\label{cor.autodualwithfixedhol}
Let $E$ be a complex vector bundle on an irreducible holomorphic symplectic manifold with a fixed k\"ahler class. Suppose $E$ is endowed with a holomorphic structure $\debar$ and a hyper-holomorphic connection $\nabla=\nabla^{1,0}+\debar$. Then
\begin{enumerate}
    \item for every holomorphic deformation $\,(E,\debar')$ of $\,(E,\debar)$ there exists an autodual connection $\nabla'$ whose $(0,1)$-part is $\debar'$,
    \item there exists a $1\colon 1$ correspondence
    \[ \left\{ \mbox{First order deformations of $\debar$} \right\}\leftrightarrow \left\{\begin{aligned}
    \mbox{First order} &\mbox{ autodual deformations of $\,\nabla\,$}\\
    &\mbox{whose $(0,1)$-part is $\debar$}
    \end{aligned}  \right\} \; . \]
\end{enumerate}
\begin{proof}
Let us start from item $(1)$. Infinitesimal deformations of the holomorphic vector bundle $(E,\debar)$ are controlled by the DG Lie algebra
\[ \left(A^{0,\ast}(End(E)),[\debar,-],[-,-]\right) \]
by Proposition~\ref{prop:deformations with dgL}. On the other hand, deformations of an autodual connection are controlled by the DG Lie algebra
\[ \left( \qA^{\ast}(End(E))\,,\, x\,[\nabla^{0,1}_J,-]+y\, [\debar,-] \,,\,[-,-] \right) \]
by Theorem~\ref{thm:iso of def}.
The projection
\[ \hat{\pi}_y\colon\qA^{\ast}(End(E))\longrightarrow A^{0,\ast}(End(E)) \]
defined as the evaluation at $x=0,y=1$ is a morphism of DG Lie algebras, so that we have the induced natural transformation
\[ \Def_{\nabla}\cong\Def_{\qA^{\ast}(End(E))}\xrightarrow{\pi_y}\Def_{(E,\debar)} \; .\]
The claim is equivalent to prove that $\pi_y$ is surjective, which is exactly the statement of Proposition~\ref{prop:surjectivity}.

Let us now turn to item $(2)$. First order deformations are parametrised by elements in the tangent space of a deformation functor. Recall that by Proposition~\ref{prop.autodualVSdefor} there is a natural transformation 
\[ \Def_\nabla\cong\Def_{\qA^{\ast}(End(E))} \longrightarrow \Def_{\left(A^{0,\ast}(End(E))\,,\,[\nabla^{0,1},-]\right)} \times \Def_{\left(A^{0,\ast}(End(E))\,,\,[\nabla^{0,1}_J,-]\right)} \]
that is an isomorphism on tangent spaces, and moreover 
\[ T^1\Def_{\left(A^{0,\ast}(End(E))\,,\,[\nabla^{0,1},-]\right)} \cong T^1\Def_{\left(A^{0,\ast}(End(E))\,,\,[\nabla^{0,1}_J,-]\right)} \; . \]
The claim follows.
\end{proof}
\end{cor}

%%%%%%%%%%%%%%%%%%%%%%%%%%%%%%%%%%
%%%%%%%%%%%%%%%%%%%%%%%%
%%%%%%%%%%%%

\end{document}